\documentclass[11pt]{article}
\usepackage[utf8]{inputenc}
\usepackage{amsmath,amssymb,amsthm}
\usepackage{mathtools}
\usepackage{enumerate}
\usepackage{tikz}
\usetikzlibrary{arrows.meta,calc}
\usepackage{lipsum}

\usepackage{graphicx} 
\usepackage{tikz-3dplot}
\usepackage{comment}
\usepackage{url} 

\newtheorem{theorem}{Theorem}[section]
\newtheorem{lemma}[theorem]{Lemma}
\newtheorem{definition}[theorem]{Definition}

\newtheorem{remark}[theorem]{Remark}
\newtheorem{corollary}[theorem]{Corollary}
\newtheorem{theorem/definition}[theorem]{Satz/Definition}
\newtheorem{proposition}[theorem]{Proposition}

\usepackage{xcolor}
\definecolor{darkgreen}{rgb}{0,0.7,0}

\usepackage{xcolor}
\usepackage{etoolbox}
\usepackage[
  colorlinks=true,
  linkcolor=darkgreen,
  citecolor=darkgreen,
  urlcolor=black
]{hyperref}

\makeatletter
\patchcmd{\tableofcontents}
  {\@starttoc{toc}}
  {\begingroup\hypersetup{linkcolor=black}\@starttoc{toc}\endgroup}
  {}{}
\makeatother

\newcommand{\R}{\mathbb R}

\newcommand{\N}{\mathbb{N}}

\newcommand{\C}{{\cal C}}

\newcommand{\op}{\operatorname}

\newcommand{\ep}{\varepsilon}

\newcommand{\overbar}[1]{\mkern 1.5mu\overline{\mkern-1.5mu#1\mkern-1.5mu}\mkern 1.5mu}
\newcommand{\Hy}{\mathbb{H}}

\numberwithin{equation}{section}

\begin{document}

\title{Lipschitz continuity of the cut time for globally hyperbolic spacetimes}
\author{{\sc Alec Metsch$^{1}$} \\[2ex]
      $^{1}$ Universit\"at zu K\"oln, Institut f\"ur Mathematik, Weyertal 86-90, \\
      D\,-\,50931 K\"oln, Germany \\
      email: ametsch@math.uni-koeln.de \\[1ex]
      }

\maketitle

\begin{abstract}
\noindent
We prove that, for a fixed point in a globally hyperbolic spacetime, the focalization time is locally Lipschitz continuous on the open subset of the future causal cone where it is finite. We also show that the cut time is locally Lipschitz continuous in a neighborhood of any timelike tangent vector whose associated geodesic is defined at least up to (and including) its cut time.
Furthermore, we derive quantitative estimates for the Lipschitz constant near the null cone and provide a criterion ensuring that the Lipschitz property extends to null directions. As a consequence, we show that the cut locus of a point has Hausdorff codimension at least $1$.
These results extend classical Lipschitz continuity results for complete Riemannian manifolds due to Itoh-Tanaka and Li-Nirenberg. Our approach follows the method of Itoh-Tanaka, suitably adapted to the Lorentzian setting.
\end{abstract}

\tableofcontents

\section{Introduction}

The cut locus plays a fundamental role in both Riemannian and Lorentzian geometry. For a complete Riemannian manifold, it is, for instance, the smallest set such that the squared Riemannian distance function is smooth on its complement (see, e.g., \cite{Cordero/McCann/Schmuckenschlaeger}); a similar statement holds in the Lorentzian setting. 

The cut locus was introduced by Whitehead \cite{Whitehead} in the Riemannian setting and has since been the subject of extensive research \cite{Sakai, Takeuchi, Gluck/Singer, Hebda, Itoh}. Among the resulting developments, Itoh and Tanaka \cite{Itoh/Tanaka} and, independently, Li and Nirenberg \cite{Li/Nirenberg} established the Lipschitz continuity of the cut time (also called distance function to the cut locus) of a point on its domain\footnote{By this we mean the open(!) set on which the function is finite}. In \cite{Itoh/Tanaka}, the first step towards this result was to establish the Lipschitz continuity of the focalization time (also called distance function to the conjugate locus) on its domain.

In contrast, in the Lorentzian setting no comparable regularity result for the cut time is known. Here, the role of the Riemannian distance function is played by the time separation function. So far, only continuity has been established \cite{Ehrlich}. Following the general strategy of \cite{Itoh/Tanaka}, our main result extends the Lipschitz property of the cut time $\rho$ of a point \eqref{eqab} to globally hyperbolic spacetimes – see Theorem \ref{thmd} (but note the discussion below). In analogy with the Riemannian approach, we first prove the Lipschitz continuity of the focalization time $\lambda$ \eqref{eqac} as an intermediate step (Theorem \ref{thmc}).

Theorem \ref{thmc} can be viewed as the direct analogue of the corresponding Riemannian result: it shows that $\lambda$ is locally Lipschitz continuous on its relatively open domain (with respect to the future causal cone), and in particular also along null directions. The situation for Theorem \ref{thmd} is more delicate. First, it is not clear if the domain of $\rho$ is open; this may occur when the cut time of a tangent vector is finite, while the associated geodesic is not defined up to the cut time. Even the continuity result of \cite{Ehrlich} does not cover this situation. This issue is essentially an artefact of potential geodesic incompleteness (compare the analogous situation for non-complete Riemannian manifolds) and can be avoided by restricting attention to future-directed causal vectors for which $\rho$ is finite and the geodesic is defined up to the cut time.
 A second, more substantial difficulty is that we are only able to establish local Lipschitz continuity in timelike directions, whereas in directions approaching the null cone we obtain only estimates on the blow-up of the Lipschitz constant. However, if the cut point of a null vector occurs before the first conjugate point, then the Lipschitz property holds without restriction in a neighbourhood of this vector – see Corollary \ref{thmd}. In either case, it follows as in the Riemannian setting that the cut locus of a point has Hausdorff codimension at least $1$ (Corollary \ref{thmc}).
 
Whereas the results in \cite{Itoh/Tanaka} and \cite{Li/Nirenberg} cover the cut locus of embedded submanifolds, we will restrict ourselves only to the cut locus of single points. We
expect, however, that by combining the arguments of the present paper with
those in \cite{Itoh/Tanaka}, the result extends to suitable classes of spacelike submanifolds
without substantial additional difficulty.

Let us also mention the works of \cite{Figalli/Rifford/Villani}, where the authors prove local semiconcavity of the cut time in dimension two under appropriate convexity assumptions on the so-called tangent focal locus, and \cite{Castelpietra/Rifford}, where the authors establish local semiconcavity of the focalization time and local Lipschitz continuity of the cut time for viscosity solutions of Hamilton-Jacobi equations associated with a certain class of Hamiltonians.

\subsection{Setting and results}

\noindent To state our results, we now describe the precise setting and introduce the necessary definitions. Throughout, $(M,g)$ denotes an $n$-dimensional globally hyperbolic spacetime with time separation function $d$. For a point $x\in M$, we will denote by $\C_x$ the future cone, i.e.\ the set of all future directed causal vectors, and we set $|v|_g:=\sqrt{|g(v,v)|}$, $v\in T_xM$. The geodesic emerging from $x$ with initial velocity $v\in T_xM$ is denoted by $c_v$.

\begin{definition}[Cut time, cut locus and focalization time]\rm
Fix $x\in M$.
\begin{enumerate}[(i)]
    \item  The \emph{cut time} is the function $\rho:\C_x\to (0,\infty]$ defined by
    \begin{align}
    \rho(v):=\sup\{t\geq 0\mid d(x,c_v(t))=t|v|_g\}.\label{eqab}
    \end{align}
    If $c_v$ only exists up to some $t>0$, then $\rho(v)\leq t$.
    The point $c_v(\rho(v))$ – if it exists – is called a \emph{cut point of} $x$. The set of cut points is denoted by $\op{Cut}(x)$.
    \item The \emph{focalization time} is the function $\lambda:\C_x\to (0,\infty]$ defined by
    \begin{align}
        \lambda(v):=\inf\{t\geq 0\mid c_v(t) \text{ is conjugate to $x$ along } c_v\}. \label{eqac}
    \end{align}
\end{enumerate}
\end{definition}
$\rho$ and $\lambda$ are easily seen to be $(-1)$-homogeneous. Since $(M,g)$ is globally hyperbolic, it is well-known that $\rho$ is continuous at $v$ unless $\rho(v)$ is finite and $c_v(\rho(v))$ is not defined \cite{Ehrlich}. The proof is analogous to the (complete) Riemannian case (where the cut time is everywhere continuous since geodesics are defined globally) once Theorem \ref{thme} has been established – which, however, is tedious for null directions. Similarly, $\lambda$ is continuous at $v$ unless $\liminf_{\C_x\ni w\to v} \lambda(w)<\infty$ but $c_v$ is not defined up to that point. As I am not aware of any (simple) proof of this statement, we include one in the appendix.

This proof forms part of a short survey collecting classical results on the relation between conjugate points, the Lorentzian index form, and the maximality of geodesics, which will be used frequently in the paper. It also includes the continuity of $\rho$. Most of these results are well known and can be found in \cite{Ehrlich} and \cite{ONeill}; however, for the convenience of the reader, and since the existing literature does not provide exactly the form needed here, we present them in a version tailored to our specific applications.
\medskip

We now state our results. Concerning the continuity of $\rho$, we start with a simple but useful and interesting proposition. Although the result is perhaps well known (at least in the Riemannian setting), we have not found an explicit reference.  We therefore state it here and include a proof in the appendix for completeness. The result can be viewed as a consequence of the first variation formula.

\begin{proposition}[Semiconvavity of the cut time] \label{li}
    Fix $x\in M$ and let $v_*\in T_xM$ be future directed timelike. Suppose that $\rho(v_*)<\lambda(v_*)$ and that $c_{v_*}(\rho(v_*))$ is defined. Then $\rho$ is locally semiconcave near $v_*$.
\end{proposition}

The following two results extend Theorems A and B of \cite{Itoh/Tanaka} to the Lorentzian setting. 

\begin{theorem}[Lipschitz continuity of the focalization time]\label{thmc}
    Fix $x\in M$. The focalization time $\lambda$ is locally Lipschitz on its relatively open domain $\lambda^{-1}(\R)$. 
\end{theorem}

Local semiconcavity of $\lambda$ can also be proved (see \cite{Castelpietra/Rifford} for the Riemannian case). As our main focus is the regularity of the cut time, we omit the proof.

\begin{theorem}[Lipschitz continuity of the cut time]\label{thmd}
    Fix $x\in M$ and let $v_*\in T_{x}M$ be future directed causal. Suppose that $\rho(v_*)<\infty$ and that $c_{v_*}(\rho(v_*))$ is defined. Let $|\cdot|_x$ be any norm on $T_xM$. Then there exists a constant $C>0$ and an open neighbourhood $U$ of $v_*$ such that 
\[
    \forall v,w\in U\cap \op{int}(\C_{x}):\ |\rho(v)-\rho(w)| 
    \leq 
    \frac{C|v-w|_{x}}{\min\{|v|^2_g,|w|^2_g\}}.
\]
In particular, if $v_*$ is timelike, then $\rho$ is Lipschitz in a neighbourhood of $v_*$.
\end{theorem}

\begin{corollary}[Hausdorff dimension of the cut locus]\label{cor5}
    The cut locus of a point $x\in M$ has Hausdorff dimension at most $n-1$.
\end{corollary}
\begin{proof}
    Fix any norm $|\cdot|_x$ on $T_xM$, and define the tangential cut locus
    \[
        \op{Cut}_{tan}(x) = \{\rho(v)v\mid v\in S_xM \text{ future directed causal},\ c_v(\rho(v)) \text{ is defined}\},
    \]
    where $S_xM$ denotes the unit sphere w.r.t.\ the $|\cdot|_x$.
    By the above theorem, $\op{Cut}_{tan}(x)\cap \op{int}(\C_{x})$ a local Lipschitz graph over an open subset (possibly empty) of the future unit sphere $S_xM\cap \op{int}(\C_x)$, and thus has Hausdorff dimension $0$ or $n-1$. Moreover, $\op{Cut}_{tan}(x)\backslash \op{int}(\C_{x})$ is contained in the $(n-1)$-dimensional future null cone. Thus, $\op{Cut}_{tan}(x)$ has Hausdorff dimension at most $n-1$, and hence so does $\op{Cut}(x)=\exp_{x}(\op{Cut}_{tan}(x))$.
\end{proof}

Note that Theorem \ref{thmc} claims Lipschitz continuity also around null directions. Unfortunately, not the same can be said about Theorem \ref{thmd} (although I am not aware of any counterexamples). However, if the cut point of a future null vector $v_*$ comes before the first conjugate point, one can refine the proof of Theorem \ref{thmd} to prove Lipschitz continuity around $v_*$.

\begin{corollary}[Lipschitz continuity to the null cone]\label{cor3}
    Fix $x\in M$ and let $v_*\in T_{x}M$ be future directed null. Suppose that $\rho(v_*)<\lambda(v_*)$ and that $c_{v_*}(\rho(v_*))$ is defined. Then $\rho$ is locally Lipschitz near $v_*$.
\end{corollary}

The plan of this paper is as follows. In Section \ref{preliminaries}, we describe the setting and introduce the necessary definitions and notation. In Sections \ref{sec: Lipsch. focal.} and \ref{sec: Lipsch. cut}, we prove the Lipschitz continuity of the focalization and cut time, respectively. Corollary \ref{cor3} is proved in Section \ref{sec: ext. to null}. Appendix \ref{App: A} contains a survey of classical results on cut points and conjugate points, Appendix \ref{App: B} contains the proof of Proposition \ref{li}, and Appendix \ref{App: C} contains a deferred proof.

\section{Preliminaries}\label{preliminaries}
In this paper, $(M,g)$ always denotes a globally hyperbolic spacetime. We assume the reader is familiar with the basic theory and standard notions of Lorentzian geometry. More specific material on conjugate points, Lorentzian index forms, and related topics is recalled in the Appendix \ref{App: A}. In our convention, $g$ has signature $(-,+,\ldots,+)$. Future directed (resp.\ past directed) curves are assumed to be piecewise smooth. We use the standard notation $J^+$ (resp.\ $I^+$) for the sets of points that are causally (resp.\ chronologically) related.

For $x\in M$, we denote by $\C_x\subseteq T_xM$ the cone of future directed causal vectors. Note that $\C_x$ is closed ($0$ being a future directed causal vector), and that $\op{int}(\C_x)$ is the set of future directed timelike vectors. We also set $\C:=\{(x,v)\in TM\mid v\in \C_x\}$. We can equip $M$ with a complete Riemannian metric $h$, which will be fixed from now on. All balls $B_r(x)$, $x\in M$, $r>0$, as well as all Lipschitz constants, are understood with respect to $h$. The $h$- and $g$-norm of $v\in T_xM$ are denoted by $|v|_h$ and $|v|_g:=\sqrt{|g(v,v)|}$. If $x$ is fixed, we denote by $c_v$ the $g$-geodesic starting at $x$ with initial velocity $v \in T_xM$.

The Lorentzian time separation function is denoted by $d$:
\[
    d(x,y):=
    \begin{cases}
        \sup_\gamma {\ell}_g(\gamma),&\text{ if } (x,y)\in J^+,
        \\
        0,& \text{ else,}
    \end{cases}
\]
where the supremum is taken over all future directed causal curves $\gamma:[a,b]\to M$ connecting $x$ to $y$ and the (Lorentzian) length ${\ell}_g(\gamma)$ is defined as
\[
    \int_a^b \sqrt{|g(\dot \gamma(t),\dot \gamma(t))|}\, dt.
\]
Since $M$ is globally hyperbolic, it is well-known that $d$ is continuous, and that each pair $(x,y)\in J^+$ can
be connected by a maximizing geodesic (\cite{ONeill}, Chapter 14, Proposition 19 and Lemma 22).
In fact,  it can be shown that $d$ is locally semiconvex on $I^+$ (see below for the definition). Moreover, for each maximizing geodesic $c:[a,b]\to M$ from $x$ to $y\in I^+(x)$, a sub-differential of $d(x,\cdot)$ at $y$ is given by
\begin{align}
    -\Big \langle \frac{\dot c(b)}{|\dot c(b)|_g},\cdot\Big\rangle_g\in T_y^*M. \label{eqsub}
\end{align}
In other words, $-\frac{\dot c(b)}{|\dot c(b)|_g}$ is a sub-gradient. These results are proved for instance in \cite{McCann2} or \cite{Metsch3}.

Recall that a function defined on an open subset of $M$ is \emph{locally semiconvex} if it is so when computed in some (hence any) chart. The set of \emph{sub-differentials} of a function $f:U\subseteq M\to \R$ at $x\in U$ is defined as
\[
    \partial^-f(x) := \{p\in T_x^*M\mid f(\exp_x(v))\geq f(x)+p(v)+o(|v|_h)\},
\]
and \emph{sub-gradients} are defined by duality via the canonical isomorphism $T_xM\to T_x^*M,\ v\mapsto g(v,\cdot)$.

If $f$ is locally semiconvex, then $\partial^-f(x)\neq \emptyset$ for all $x\in U$. In fact, there exists an open neighbourhood $V$ of $x$ and $K,r>0$ such that, for any $y\in V$, any $v\in T_yM$ with $|v|_h<r$ and any sub-differential $p$ of $f$ at $y$, we have
\begin{align}
    f(\exp_y(v))\geq f(y)+p(v)-K|v|_h^2. \label{eqao}
\end{align}
It is well-known that locally semiconvex functions are locally Lipschitz continuous.\\

In the present paper, we frequently use the continuity properties of $\lambda$ and $\rho$ in the sense of Corollary \ref{cor2} and Lemma \ref{lq}, as well as the compactness property of maximizing geodesics in the sense of Lemma \ref{lr}.

\section{Lipschitz continuity of the focalization time}\label{sec: Lipsch. focal.}

In this section, we will prove Theorem \ref{thmc}. In order to so, fix $x\in M$ and $v_*\in T_{x}M$ future directed causal with $\lambda(v_*)<\infty$. Recall that $\lambda$ is continuous in a neighbourhood of $v_*$ (Lemma \ref{cor2}). Combining the continuity of $\lambda$ with a connectedness argument, it suffices to prove the following in order to establish Theorem \ref{thmc}.
\begin{quote}
There exists an open neighbourhood $U\subseteq T_{x}M$ of $v_*$ and a constant $C>0$ such that, for all $v\in U$ future directed timelike, we have:
\begin{align}
\lambda(w)-\lambda(v)\leq C|v-w|_h \label{eqam}
\end{align}
holds for any $w\in U$ future directed timelike and sufficiently close to $v$.
\end{quote}

We define the set $U$ as follows. Let $U\Subset T_{x}M$ be an open neighbourhood of $v_*$, and let $r>0$ such that $\lambda(v)<r$ and $c_v(r)$ is defined for all $v\in \overbar U$.
 Let $E_i$, $i=1,...,n$, be an orthonormal basis of $T_{x}M$ w.r.t.\ $g$ such that $E_1$ is future directed timelike. Let $E_i(t,v)$, $t\in [0,r]$, $v\in \overbar U$, denote the parallel transport of $E_i$ along the geodesic $c_v$. 
\medskip 

For $v \in \overbar U$, the geodesic $c_v$ is defined on $[0,r]$. Hence, whenever $s\in [0,r]$, we can consider the Lorentzian index form $I_s : V_s^\perp(c_v)\times V_s^\perp(c_v)\to \mathbb{R}$ along $c_v$ on $[0,s]$ (cf.\ \eqref{eqh}). Note that elements of $V_s^\perp(c_v)$ are only assumed to be piecewise smooth. We suppress the dependence on $v$ in the notation of $I_s$. 

The following lemma corresponds to (1.17) in \cite{Itoh/Tanaka}.

\begin{lemma}[Lipschitz continuity of the index form]\label{lb}
    For every $R>0$, there exists a constant $C(R)>0$ with the following property.

Let $s\in[0,r]$ and $v,w\in \overbar U\cap\C_x$. 
Let $X\in V_s^\perp(c_v)$ and $Y\in V_s^\perp(c_w)$ be of the form
\[
    X(t)=\sum_{i=1}^n a_i(t) E_i(t,v), 
    \qquad 
    Y(t)=\sum_{i=1}^n b_i(t) E_i(t,w).
\]
Assume that, for all $i=1,\dots,n$ and for every common differentiability point $t\in[0,s]$ of $X$ and $Y$, one has
\[
    |a_i(t)|,\,|\dot a_i(t)| \le R,
    \qquad
    |a_i(t)-b_i(t)|,\,|\dot a_i(t)-\dot b_i(t)|
    \le R |v-w|_h .
\]
Then
\[
    |I_s(X,X)-I_s(Y,Y)| \le C(R)\,|v-w|_h .
\] 
Note that $I_s$ denotes both the Lorentzian index form along $c_v$ and along $c_w$.
\end{lemma}
\begin{proof}
    The function $F:\R^n\times \R^n \times \overbar U\times [0,r]\to \R$ defined by
    \[
        F(x,y,v,t):= -\sum_{i,j=1}^n \Big(\ep_i\delta_{ij} y_i^2- x_i x_j \Bigl\langle R(E_i(t,v),\dot c_v,\dot c_v),E_j(t,v)\Bigl\rangle_g\Big),
    \]
    is smooth (i.e.\ can be smoothly extended to a suitable open neighbourhood) and hence Lipschitz when $x$ and $y$ are restricted to the compact  ball of radius $R+R\op{diam}(U)$. Here, $\ep_i$ is $-1$ for $i=1$ and $1$ otherwise, and $\delta_{ij}$ is the Kronecker delta. Now note that, by \eqref{eqh},
    \[
        I_s(X,X)-I_s(Y,Y) = \int_0^s F(a(t),\dot a(t),v,t) - F(b(t),\dot b(t),w,t)\, dt,
    \]
    where $a(t):=(a_1(t),...,a_n(t))$, and $b(t):=(b_1(t),...,b_n(t))$. The lemma follows easily.
\end{proof}

\begin{remark}
    \rm 
    \begin{enumerate}[(i)]
        \item 
        Let $J(\cdot,v,w)$ denote the Jacobi field along $c_v$ with 
        \[
        J(0,v,w)=0 \quad \text{ and } \quad \dot J(0,v,w)=w.
        \]
        Since the differential of $\exp_{x}$ at $0$ is the identity, there exists $\ep>0$ such that $J(t,v,w)\neq 0$ for all $v\in \overbar U$, $|w|_h=1$ and $t\in [0,\ep]$ (Remark \ref{hdsdoa}). Since moreover $J$ and $\dot J$ depend continuously on $v$ and $w$ (Lemma \ref{lc}), it follows that the following supremum and infimum are attained:
        \begin{align*}
            &\sup\{|J(t,v,w)|_h+|\dot J(t,v,w)|_h\mid t\in [0,r],\ v\in \overbar U,\ |w|_h=1\},
            \\[5pt]
            &\inf\{|\dot J(t,v,w)|_g\mid  t\in (0,r],\ v\in \overbar U,\ |w|_h=1,\ J(t,v,w)=0\}.
        \end{align*}
        If $J$ denotes any Jacobi field as in the infimum, then it is easy to see\footnote{For any Jacobi field $X$ along a geodesic $c$, $\langle X,\dot c\rangle_g$ is affine linear.} that $\dot J\perp \dot c_v$. Moreover, $\dot J(t)$ cannot be $0$, since otherwise $J\equiv 0$. Finally, $\dot J(t)$ cannot be parallel to $\dot c_v(t)$ (which, by $\dot J\perp \dot c_v$, was only possible if $c_v$ is null) since otherwise $J(s)=\mu (s-t)\dot c_v(t)$ for some $\mu\in \mathbb{R}$, which would imply $J(0)\neq 0$. Thus, $\dot J(t)$ is spacelike and hence $|\dot J(t)|_g\neq 0$.
        \item 
        We choose $C_1$ such that the supremum (resp.\ infimum) is bounded from above (resp.\ below) by $C_1$ (resp.\ $1/C_1$) and such that 
        \begin{align}
            |E_i(t,v)|_h\leq C_1  \quad \text{ and }  \quad |\langle u,w\rangle_g|\leq C_1\, |u|_h \, |w|_h \label{eqal}
        \end{align}
        for all $t\in [0,r]$, $v\in \overbar U$, $u,w\in T_yM$ and $y=c_v(t)$. The existence of $C_1$ follows from a simple compactness argument.
    \end{enumerate}
\end{remark}

The following theorem is the Lorentzian analogue of Lemma 1.3 in \cite{Itoh/Tanaka}.

\begin{theorem}[Linear growth of the index form]\label{thmb}
    For any $v\in \overbar U$, there exists $s(v)>0$ with $\lambda(v)+s(v)\leq r$ such that, for any $s\in [0,s(v)]$, there is a vector field $X_s \in V_{\lambda(v)+s}^\perp(c_v)$, say
    \[
        X_s(t)=\sum_{i=1}^n a_i(s,t)E_i(t,v),
    \]
    with $I_{\lambda(v)+s}(X_s,X_s)\geq C_1^{-2}s/2$ and $|a_i(s,t)|+|\dot a_i(s,t)|\leq 4C_1^3$ for $i=1,...,n$. Here, $C_1$ is an in the remark above.
    
\end{theorem}

\begin{proof}
    Fix $v\in U$. Let $a=a(v)\in [0,\lambda(v))$ such that $c_{v}(a)\in V$, where $V$ is a convex neighbourhood of $c_v(\lambda(v))$. Let $X$ be a Jacobi field along $c_{v}$ vanishing at $0$ and $\lambda(v)$ with $|\dot X(0)|_h=1$. Define $s(v)>0$ such that
    \[
        \lambda(v)+s\leq r \quad \text{ and } \quad c_v(\lambda(v)+s)\in V \quad \forall s\in [0,s(v)].
    \]
    Since $V$ is convex, there exists a Jacobi field $Y_s$ along $c_v$ with $Y_s(a)=X(a)$ and $Y_s(\lambda(v)+s)=0$ (Lemma \ref{lc}). Consider the broken Jacobi field
    \begin{align*}
      X_s:[0,\lambda(v)+s]\to TM,\ X_s(t)=
      \begin{cases}
          X(t),\ &\text{if } t\leq a,
          \\
          Y_s(t),\ &\text{otherwise}.
      \end{cases}
    \end{align*}
    We claim that $X_s\in V_{\lambda(v)+s}^\perp(c_v)$. Indeed, since $X(0)=X(\lambda(v))=0$, we have $\langle X, \dot c_v\rangle_g \equiv 0$. Moreover, $Y_s(\lambda(v)+s)=0$ and, by definition of $Y_s$ and since $\langle X, \dot c_v\rangle_g \equiv 0$, also $\langle Y_s(a),\dot c_v(a)\rangle_g=0$. Hence, $\langle Y_s,c_v\rangle_g\equiv 0$, since $Y_s$ is a Jacobi field. This shows $X_s\in V_{\lambda(v)+s}^\perp(c_v)$ and concludes the first part of the theorem.
    
    Being $X$ and $Y_s$ Jacobi fields, using first the formula of the index form \eqref{eqh} and then the standard Lemma \ref{ld}, we get
    \begin{align*}
         I_{\lambda(v)+s}(X_s,X_s) &= \big\langle \dot X_s(a^+),X_s(a)\big\rangle_g -\big\langle \dot X_s(a^-),X_s(a)\big\rangle_g
        \\[5pt]
        &=\big\langle \dot Y_s(a),X(a)\big\rangle_g -\big\langle \dot X(a),Y_s(a)\big\rangle_g
        \\[5pt]
        & = \big\langle \dot Y_s(\lambda(v)+s),X(\lambda(v)+s)\big\rangle_g -\big\langle \dot X(\lambda(v)+s),Y_s(\lambda(v)+s)\big\rangle_g.
    \end{align*}
    The second term vanishes by definition of $Y_s$. Since ${Y_s}$ clearly converges in the $C^2_{loc}$-topology to $X$ as $s\to 0$ (Lemma \ref{lc}), and since $X(\lambda(v))=0$, we compute 
    \[
        \lim_{s\to 0} \inf_{t\in [\lambda(v),\lambda(v)+s]}\frac{d}{dt}
        \big\langle \dot Y_s(t),X(t)\big\rangle_g
        = |\dot X(\lambda(v))|_g^2 \geq C_1^{-2}.
    \]
    Hence, we obtain that
    \[
         I_{\lambda(v)+s}(X_s,X_s) \geq  C_1^{-2}s/2
    \]
    for small $s$, concluding the proof of the second part of the theorem.

    Finally, note that, since $|X(t)|_h+|\dot X(t)|_h\leq C_1$ for $t\in [0,r]$, and since $Y_s$ converges in the $C_{loc}^1$-topology to $X$, we have $|X_s(t)|_h+|\dot X_s(t)|_h\leq 2C_1$ for small $s$. Thus, by \eqref{eqal},
    \[
    |a_i(s,t)|+|\dot a_i(s,t)|=|\langle Y_s(t),E_i(t,v)\rangle_g|+|\langle \dot Y_s(t),E_i(t,v)\rangle_g| \leq 4C_1^3. \qedhere
    \]
    \end{proof}

\begin{lemma}
    There exists a constant $C_2>0$ such that, for any $t\in [0,r], v,w\in \overbar U$ future directed causal and $a\in \R^n$ with $|a_i|\leq 4C_1^3$, we have
    \[
         \sum_{i=2}^n a_i\Big(\frac{\langle E_i(t,v),\dot c_v(t)\rangle_g}{\langle \dot c_v(t),E_1(t,v)\rangle_g} 
        -
        \frac{\langle E_i(t,w),\dot c_w(t)\rangle_g}{\langle \dot c_w(t),E_1(t,w)\rangle_g}  \Big)
        \leq C_2 |v-w|_h.
    \]
\end{lemma}
\begin{proof}
    This follows from the smooth dependence of the sum on $v$ and $w$.
\end{proof}

We are now in position to prove Theorem \ref{thmc}. The proof follows the approach of \cite{Itoh/Tanaka}: For $v\in U$ timelike and $w\in U$ sufficiently close to $v$, we use $X_s$ from above to construct a vector field $Y_s$ along $c_w$ such that $I_{\lambda(v)+s}(Y_s,Y_s)>0$. This implies $\lambda(w)\leq \lambda(v)+s$. We must show that $s$ can be bounded linearly in terms of $|v-w|_h$.
    
\begin{proof}[Proof of Theorem \ref{thmc}] 
    
    Set $R:=\max\{4C_1^3,C_2\}$ and $C:=4C(R)C_1^2$, where $C(R)$ is as in Lemma \ref{lb}. Let us show \eqref{eqam}.
        
    Fix $v\in U$ future directed timelike and define $X_s$ as in Theorem \ref{thmb} for $s\leq s(v)$. Since $X_s\in V_{\lambda(v)+s}^\perp(c_v)$, we have $\langle X_s,\dot c_v\rangle_g\equiv 0$, and thus
    \[
        a_1(s,t) = \frac{-1}{\langle \dot c_v(t),E_1(t,v)\rangle_g} \sum_{i=2}^n a_i(s,t) \langle E_i(t,v),\dot c_v(t)\rangle_g.
    \]
    For $w\in U\cap \op{int}(\C_{x})$, we define
    \[
        Y_s(t) := \sum_{i=1}^n b_i(s,t) E_i(t,w)
    \]
    with $b_i=a_i$ for $i\geq 2$ and $b_1$ such that $Y_s\in V_{\lambda(v)+s}^\perp(c_w)$, i.e.\ 
    \begin{align}
        b_1(s,t):=\frac{-1}{\langle \dot c_w(t),E_1(t,w)\rangle_g}\sum_{i=2}^n a_i(s,t) \langle E_i(t,w),\dot c_w(t)\rangle_g. \label{eqi}
    \end{align}
    Denote by $\dot a(s,t)$ and $\dot b(s,t)$ the derivative w.r.t.\ $t$. Whenever $X_s$ is differentiable, it holds
    \begin{align}
        \dot b_1(s,t):=-\frac{1}{\langle \dot c_w(t),E_1(t,w)\rangle_g}\sum_{i=2}^n \dot a_i(s,t) \langle E_i(t,w),\dot c_w(t)\rangle_g, \label{eqj}
    \end{align}
    by the parallelity of $\dot c_w$ and $E_i$. 
    Using the bounds from Theorem \ref{thmb}, we have
    \[
        |a_1(s,t)|,|\dot a_1(s,t)|\leq 4C_1^3 \leq R,
    \]
    so that the above lemma implies
    \[
        |b_1(s,t)-a_1(s,t)|,|\dot b_1(s,t)-\dot a_1(s,t)|\leq C_2|v-w|_h \leq R|v-w|_h.
    \]
     If $w\neq v$ is sufficiently close to $v$ so that $s:=C|v-w|_h\leq s(v)$, Lemma \ref{lb} and Theorem \ref{thmb} imply that the Lorentzian index form $I_{\lambda(v)+s}$ along $c_w$ is not negative semidefinite:
    \[
        I_{\lambda(v)+s}(Y_s,Y_s) \geq I_{\lambda(v)+s}(X_s,X_s)-C(R) |v-w|_h \geq \frac{C_1^{-2}}{2}s-C(R)|v-w|_h>0.
    \]
    Thus, as $w$ is timelike, Corollary \ref{prod} implies $\lambda(w)\leq \lambda(v)+C|v-w|_h$.
\end{proof}

\section{Lipschitz continuity of the cut time}\label{sec: Lipsch. cut}

In this section, we prove Theorem \ref{thmd}. In order to do so, fix $x\in M$ and $v_*\in T_{x}M$ future directed causal such that $\rho(v_*)<\infty$ and $y_*:=c_{v_*}(\rho(v_*))$ is defined. 

The goal is to prove that there exists a constant $C>0$ and an open neighbourhood $U$ of $v_*$ such that 
\begin{align}
    \forall v_0,v_1\in U\cap \op{int}(\C_{x}):\ |\rho(v_1)-\rho(v_0)| 
    \leq 
    \frac{C|v_1-v_0|_h}{\min\{|v_0|_g^2,|v_1|_g^2\}}.\label{eqag2}
\end{align} 

As in \cite{Itoh/Tanaka}, we first show how the theorem follows from the following lemma. Lemma \ref{la} is what we will prove in the rest of the section.

\begin{lemma}[Uniform bound for the  Lipschitz constant]\label{la}
  There exists a constant $C > 0$ and an open neighbourhood $U$ of $v_*$ such that the following holds: 
Suppose that $v_0, v_1 \in U \cap \operatorname{int}(\mathcal{C}_x)$ and define
\[
v(s) := (1-s)v_0 + s v_1 .
\]
If $v(s) \in U$ and $\rho(v(s)) < \lambda(v(s))$ for all $s$, and if $\rho \circ v$ is differentiable at $s_0$, then
\begin{equation}
\left| (\rho \circ v)'(s_0) \right|
\leq
\frac{C \, |v_1 - v_0|_h}{\min\{ |v_0|_g^2, |v_1|_g^2 \}}.
\label{eqaf}
\end{equation}
    In particular, since $\rho\circ v$ is locally Lipschitz\footnote{Although we use Proposition \ref{li}, we do not need its full strength, but only that $\rho$ is locally Lipschitz continuous in a neighbourhood of $v_*$. This is a simple consequence of the the first variation formula.} by Proposition \ref{li}, \eqref{eqag2} holds for this choice of $v_0$ and $v_1$.
\end{lemma}

The following proof is as in \cite{Itoh/Tanaka}, Proof of Theorem B.

\begin{proof}[Proof of Theorem \ref{thmd}]
    If $\rho(v_*)<\lambda(v_*)$, continuity of $\rho$ (Lemma \ref{lq}) and Lemma \ref{cor2} give $\rho(v)<\lambda(v)$ for $v$ close to $v_*$. Then Lemma \ref{la} immediately yields the claim. 
    
    Otherwise, suppose that $\rho(v_*)=\lambda(v_*)$, and let $U$ and $C$ be as in the above lemma. We may assume that $U$ is convex, that $\rho$ is continuous on $U$ (Lemma \ref{lq}) and that $\lambda$ is Lipschitz on $U$ with
    \[
        \op{Lip}(\lambda_{|U})\leq \tilde C:=\frac{C}{\min\{|v_0|^2_g,|v_1|^2_g\}}.
    \]
    Fix $v_0,v_1\in U\cap \op{int}(\C_{x})$. If $\rho(v_0)=\lambda(v_0)$, then 
    \begin{align}
        \rho(v_1)  \leq \lambda(v_1) \leq \lambda(v_0)+\tilde C \, |v_1 - v_0|_h= \rho(v_0)+\tilde C \, |v_1 - v_0|_h. \label{eqba}
    \end{align}
    Otherwise, let $v(s):=(1-s)v_0+sv_1$, and define 
    \[
        I:=\{t\in [0,1]\mid \rho(v(s)) < \lambda(v(s)) \quad \forall s\in [0,t]\}.
    \]
    Obviously, $0\in I$. Let $T:=\sup(I)$. Then the preceding lemma gives
    \[
        \rho(v(T))-\rho(v_0)= \lim_{t\uparrow T} \rho(v(t))-\rho(v_0)
        \leq \frac{C\, |v(T)-v_0|_h}{\min\{|v_0|_g^2,|v(T)|_g^2\}}.
    \]
    Moreover, continuity of $\rho$ and $\lambda$ give $\rho(v(T))=\lambda(v(T))$, and thus we may argue as in the first case to obtain
    \[
        \rho(v_1) \leq \lambda(v_1) \leq \rho(v(T))+\tilde C\, |v_1-v(T)|_h.
    \]
    Using the fact that $s\mapsto |v(s)|_g$ is concave, one easily concludes that
    \[
        \rho(v_1)-\rho(v_0) \leq \tilde C\, |v_1-v_0|_h.
    \]
    Thus, this inequality holds in any case. Exchanging the roles of $v_0$ and $v_1$, we obtain \eqref{eqag2}.
\end{proof}

\subsection{Uniform estimates}

In order to prove Lemma \ref{la}, we need some uniform estimates. Notice that Part (i) of the following lemma is trivial in the Riemannian case. Part (ii) and (iii) correspond to Lemma 2.4 and 2.5 in \cite{Itoh/Tanaka}. 

\begin{lemma}[Uniform estmates]\label{lh}
    Let $U\subseteq T_{x}M$ be a precompact neighbourhood of $v_*$ such that $\rho(v)<\infty$ and $c_v(\rho(v))$ is defined for all $v\in \overbar U\cap \C_x$. Then there exists a constant $C>0$ such that the following properties hold for all $v\in \overbar U\cap \C_x$ and $w\in \C_{x}$ satisfying $\rho(v)=\rho(w)$ and $c_v(\rho(v))=c_w(\rho(w))$:
    \begin{enumerate}[(i)]
    \item $1/C\leq |\dot c_w(\rho(w))|_h\leq C$.
        \item $|v-w|_h \leq C |\dot c_v(\rho(v))-\dot c_w(\rho(w))|_h$.
        \item Let $X=J(\cdot,v,w-v)$ (recall Definition \ref{defa}). Then
        \begin{align} 
        |X(\rho(v))|_h
        \leq C\, |v-w|_h^2, \label{eqan}
        \end{align}
        and
        \begin{align}
            |\dot c_w(\rho(w))-\dot c_v(\rho(v)) - 
            \dot X(\rho(v))|\leq C\, |v-w|_h^2. \label{eqan2}
        \end{align}
    \end{enumerate}
\end{lemma}
\begin{proof}
    Recall first that the compactness of maximizing geodesics in the sense of Lemma \ref{lr} guarantees that the set 
    of all $v,w$ as in the hypothesis forms a compact subset of $T_xM^2$.
    (i) follows immediately from this observation and the continuity of $\rho$, noticing that $w\neq 0$ for any $w$ as in the hypothesis since $v\neq 0$.
    
    To prove (ii), suppose by contrary that there exist sequences $(v_k)$ and $(w_k)$ as in the hypothesis of the lemma such that
    \begin{align}
        |v_k-w_k|_h \geq k\, |\dot c_ {v_k}(\rho(v_k))-\dot c_{w_k}(\rho(w_k))|_h. \label{eqw}
    \end{align}
    By our observation from the beginning, we may pass to a subsequence such that $v_k\to v\in \overbar U$ and $w_k\to w$ with $\rho(w)=\rho(v)$ and $c_w(\rho(w))=c_v(\rho(v))$. Suppose $v\neq w$; uniqueness of geodesics yields $\dot c_v(\rho(v))\neq \dot c_w(\rho(w))$, which leads to a contradiction by taking limits in \eqref{eqw}: 
    \[
        |v-w|_h = \lim_{k\to \infty} |v_k-w_k| \geq \limsup_{k\to \infty} k\,  |\dot c_{v_k}(\rho(v_k))-\dot c_{w_k}(\rho(w_k))|_h=\infty. 
    \]
    Thus, $v=w$. Since $(c_v(\rho(v)),-\dot c_v(\rho(v)),\rho(v))\in TM\times \R$ belongs to the domain of the smooth exponential map, and since
    \[
        \frac{\partial }{\partial t}\Big|_{t=\rho(v_k)} \exp({y_k},-t\dot c_{v_k}(\rho(v_k)))= -v_k,
    \]
    as well as
    \[
        \frac{\partial }{\partial t}\Big|_{t=\rho(w_k)} \exp({y_k},-t\dot c_{w_k}(\rho(w_k)))= -w_k,
    \]
    where $y_k=c_{v_k}(\rho(v_k))=c_{w_k}(\rho(w_k))$, we derive a contradiction to \eqref{eqw} (note that $\rho(v_k)=\rho(w_k)$ is crucial here). This proves (ii).
    
    For the proof of (iii), we shall keep the argument intuitive, 
    while being not completely rigorous. First observe that our observation above concerning the compactness guarantees the expressions on the left-hand side of \eqref{eqan} and \eqref{eqan2} to remain bounded for all possible $v$ and $w$ as in the hypothesis (recall that $J(\cdot,\cdot,\cdot)$ depends smoothly on its parameters). Therefore we may assume that $|v-w|_h$ is arbitrarily small. 
    
    Define
    \[
        f(t,v):=\exp_x(tv),
    \]
    which is defined, by our hypothesis, for $v\in \overbar U$ and $t=\rho(v)$ – and hence also for all $(t,w)$ close to $(\rho(v),v)$.
    
    For $v,w$ sufficiently close as in the hypothesis, Taylor's theorem yields for any smooth function $\varphi$
    \begin{align*}
        \varphi \circ f(t,w) -\varphi \circ f(t,v) = D_v(\varphi\circ f)(v)[w-v]+r(t,v,w)
    \end{align*}
    where $|r(t,v,w)|+|\partial_t r(t,v,w)|\leq C\, |v-w|_h^2$ for some uniform constant $C$ depending only on $\varphi$.
    By definition of $X$, we may rewrite this as
    \begin{align}
        \varphi \circ f(t,w) -\varphi \circ f(t,v) 
        = \bigl\langle \nabla \varphi(f(t,v), X(t)\bigl\rangle_g 
        +r(t,v,w).  \label{eqfa}
    \end{align}
    Differentiate both sides w.r.t.\ $t$ and put $t=\rho(v)=\rho(w)$. Setting $y:=f(\rho(v),v)$, we get
    \begin{align}
        \bigl\langle \nabla \varphi(y), \dot c_w(\rho(w))-\dot c_v(\rho(v)) \bigl\rangle_g \nonumber
        &=
        \bigl\langle
        \nabla^2\varphi(y)\cdot\dot c_v(\rho(v)), 
            X(\rho(v))\bigl\rangle_g\, 
        \\[5pt]
        &+\bigl\langle 
            \nabla \varphi(y), \dot X(\rho(v))
        \bigr\rangle_g \nonumber
        \\[5pt]
        &+\partial_t r(t,v,w). \label{eqfb}
    \end{align}
    Here, $\nabla^2\varphi$ denotes the Hessian of $\varphi$. Since $\varphi$ is arbitrary, Equation \eqref{eqfa} evaluated at $t=\rho(v)$ proves the first inequality. 
    Inserting this estimate into \eqref{eqfb}, we also obtain the second.
\end{proof}

\begin{definition}[The unit hyperboloid]\rm \label{defc}
    Consider the $n$-dimensional Minkowski space $(\R^{1,n-1},\eta)$, where the metric is given by
    \[
        \eta = (-dx_1)^2 +\sum_{i=2}^{n} (dx_i)^2.
    \]
    The \emph{unit hyperboloid} in $\R^{1,n-1}$ is the complete Riemannian manifold defined as
    \[
        \Hy:=\{x\in \R^n\mid x_1>0,\ \eta(x,x)=-1\},
    \]
    where the metric is the one induced by $\eta$. For $x,y\in \Hy$, we denote its distance by $d_{\Hy}(x,y)$.
    Note that
    \begin{align*}
    \Hy 
    = \Biggl\{ 
    \begin{pmatrix} \cosh(\phi) \\ \sinh(\phi) y \end{pmatrix} 
    \;\Big|\; \phi \in [0,\infty),\ y \in S^{n-2} \subseteq  \mathbb{R}^{n-1} 
    \Biggr\} 
\end{align*}
    We call $\phi$ the \emph{hyperbolic angle} of the vector.
\end{definition}

\begin{lemma}[Hyperbolic and Euclidean point of view]\label{ll6}
    There is a constant $C>0$ such that
    \begin{align}
        |x-\tilde x|_{\op{euc}} \leq C \cosh(\phi_{max})\, d_{\Hy}(x,\tilde x) 
        \leq C\, \max\{|x|_{\op{euc}},|\tilde x|_{\op{euc}}\}\, d_{\Hy}(x,\tilde x)\label{eqah}
    \end{align}
    for all $x,\tilde x\in \R^{1,n-1}$ with $\phi_{max}$ being the maximum of the hyperbolic angles $\phi$ and $\tilde \phi$ of $x$ and $\tilde x$, respectively.  
\end{lemma}
\begin{proof}
    The second inequality is trivial. For the first, see Appendix \ref{App: C}.
\end{proof}

\subsection{Proof of Lemma \ref{la}}\label{sa}

We will now start with the proof of Lemma \ref{la}. Let $E_1(y),...,E_n(y)$ be a smooth orthonormal frame w.r.t.\ $g$ in a neighbourhood $V$ of $y_*$, with $E_1$ future directed timelike. We may modify the complete Riemannian metric $h$ such that $E_1,...,E_n$ is a local orthonormal frame for $h$ as well (i.e.\ $h(E_1,E_1)\equiv 1$). 
    
We may assume that there exists a constant $C_1>0$ such that
\begin{align}
    |\langle v,w\rangle_g|\leq C_1\, |v|_h\, |w|_h \quad \text{ for all } y\in V \text{ and } v,w\in T_yM.\label{eqap}
\end{align}
Let us construct the neighbourhood $U$ asserted by Lemma \ref{la}. Fix an open precompact neighbourhood $U$ of $v_*$ and $r>0$ such that, for all $v\in U$, the geodesic $c_v$ is defined on $[0,r]$ with $\rho(v)<r$ and $c_v(\rho(v))\in V$. We may choose $C_1$ larger such that Lemmas \ref{lh} and \ref{ll6} apply and such that
\begin{align}
    \sup\{|\dot J(t,v,w)|+|J(t,v,w)|\mid v\in \overbar U,\ |w|_h\leq 1,\ t\in [0,r]\}\leq C_1. \label{eqad}
\end{align}
Recall that $J(t,v,w)$ is the Jacobi field along $c_v$ with $J(0,v,w)=0$ and $\dot J(0,v,w)=w$.
\medskip

Fix $v_0,v_1\in U$ and $s_0$ as in the hypothesis of the lemma. Recall that 
\[
    v(s):=(1-s)v_0+sv_1.
\]
By abuse of notation, we write $\rho(s)$ for $\rho(v(s))$.

\begin{definition}\rm
In analogy with \cite{Itoh/Tanaka}, we define the continuous curve
    \[
        c(s):=c_{v(s)}(\rho(s)),\ s\in [0,1],
    \]
    as well as the vector fields along the curve $c(s)$ (resp.\ the parametrized surface $c_{v(s)}(t)$)
    \[
        e_1(s):=-\frac{\dot c_{v(s)}(\rho(s))}{|\dot c_{v(s)}(\rho(s))|_g} \quad \text{ and } \quad Y(t,s):=\frac{\partial }{\partial s} c_{v(s)}(t),\ t\in [0,r].
    \]
\end{definition}

\begin{lemma}[A first formula for the derivative]\label{ef}
    It holds 
    \[
        \dot c(s_0)-Y(\rho(s_0),s_0) = \rho'(s_0)\,  \dot c_{v(s_0)}(\rho(s_0)).
    \]
    In particular, we get
    \begin{align}
        |v(s_0)|_g\, \rho'(s_0) = \langle \dot c(s_0)-Y(\rho(s_0),s_0),e_1\rangle_g. \label{eqt}
    \end{align} 
\end{lemma}
\begin{proof}
    Consider the parametrized surface $f(s,t):=c_{v(s)}(t)$ as above. As $\rho(s)$ is differentiable at $s_0$, so is $c(s)=f(s,\rho(s))$, and
    \begin{align*}
        \dot c(s_0)-Y(\rho(s_0),s_0) 
        &= \frac{\partial f}{\partial s}(s_0,\rho(s_0))+\rho'(s_0)\, \frac{\partial f}{\partial t}(s_0,\rho(s_0))
        -\frac{\partial f}{\partial s}(s_0,\rho(s_0))
        \\[5pt]
        &=\rho'(s_0)\,  \dot c_{v(s_0)}(\rho(s_0)).
        \qedhere
    \end{align*}
    This proves the first part. It follows that
    \[
        \langle \dot c(s_0)-Y(\rho(s_0),s_0),e_1\rangle_g = |\dot c_{v(s_0)}(\rho(s_0))|_g \rho'(s_0) = |v(s_0)|_g\, \rho'(s_0),
    \]
    having used that geodesics have constant speed.
\end{proof}

For the next lemma, recall the definition of sub-gradients via the duality of tangent and cotangent vectors given by $v\mapsto \langle v,\rangle_g$. We also define the function $d_x(y):=d(x,y)$. The first part of the following Lemma corresponds to Lemma 2.1 in \cite{Itoh/Tanaka}.

\begin{lemma}[Limit sub-gradients are optimal in the first variation]
    Let $s_k\downarrow s_0$ and $v_k\in \partial^-d_{x}(c(s_k))$ be such that $v_k\to v_0\in \partial^-d_{x}(c(s_0))$.  Then
    \begin{align}
        \langle v_0,\dot c(s_0)\rangle_g = \max\{\langle v,\dot c(s_0)\rangle_g\mid v\in \partial^-d_{x}(c(s_0))\}. \label{eqq}
    \end{align}
    Moreover, there exists a second (i.e.\ distinct from $c_{v_0}$) maximizing geodesic $c_{w(s_0)}:[0,\rho(w(s_0))]\to M$ from $x$ to $c(s_0)$ such that $\rho(w(s_0))=\rho(s_0)$ and
    \[
        \dot c(s_0)\perp e_1(s_0)-e_1(w(s_0)),\ \text{where } e_1(w(s_0)):=-\frac{\dot c_{w(s_0)}(\rho(w(s_0)))}{|\dot c_{w(s_0)}(\rho(w(s_0)))|_g}.
    \]
\end{lemma}
\begin{proof}
    For all $v\in \partial^-d_{x_0}(c(s_0))$, the semiconvexity of $d$ (cf.\ \eqref{eqao}) yields the existence of $K>0$ such that, for large $k$,
    \begin{align*}
        &d(x,c(s_k)) - d(x,c(s_0))
        \geq 
        \langle v,c(s_k)-c(s_0)\rangle_g  -K|c(s_k)-c(s_0)|_h^2 \quad \text{ and }
        \\[5pt]
        &d(x,c(s_0)) - d(x,c(s_k))
        \geq 
        \langle v_k,c(s_0)-c(s_k)\rangle_g  -K|c(s_k)-c(s_0)|_h^2,
    \end{align*}
    where we use the notation $b-a:=\exp_a^{-1}(b)$ for $b$ close to $a$.
    Note that differentiability of $c$ at $s_0$ gives
    \[
        \frac{c(s_k)-c(s_0)}{s_k-s_0}\to \dot c(s_0), \quad \text{ whereas } \quad \frac{c(s_0)-c(s_k)}{s_k-s_0} \to -\dot c(s_0). 
    \]
   Thus, adding both inequalities and dividing by $s_k-s_0$, we obtain 
   \[
        0 \geq \langle v-v_0,\dot c(s_0)\rangle_g.
    \]
    This gives \eqref{eqq}.

    We prove the second part. Since $\rho(s_k)=\rho(v(s_k))<\lambda(v(s_k))$ for all $k$, Theorem \ref{thme} ensures the existence of some $w_k\in T_{x}M$ with $\rho(w_k)=\rho(s_k)$ and $w_k\neq v(s_k)$ such that $c_{w_k}:[0,\rho(w_k)]\to M$ is a  maximizing geodesic from $x$ to $c(s_k)$. By compactness of maximizing geodesics (Lemma \ref{lr}), after passing to a subsequence, $w_k$ converges to some $w(s_0)\in T_{x}M$, and $c_{w(s_0)}:[0,\rho(w(s_0))]\to M$ is a maximizing geodesic from $x$ to $c(s_0)$ with $\rho(w(s_0))= \rho(s_0)$. Since $c(s_0)$ not conjugate to $x$ along $c_{v(s_0)}$, it follows $w(s_0)\neq v(s_0)$. We then apply the first part of the lemma twice – to $v_k:=e_1(s_k)$ and to $v_k:= -\dot c_{w_k}(\rho(w_k))/|w_k|_g$ – and conclude the proof in sight of \eqref{eqq}.
\end{proof}

We are ready to prove Lemma \ref{la}. We follow the approach of \cite{Itoh/Tanaka}. Note that our analysis to get to \eqref{eqae} (corresponding to (2.23) in \cite{Itoh/Tanaka}) is faster.

\begin{proof}[Proof of Lemma \ref{la}]
    Set $C:=2C_1^9$. Let us denote $v:=v(s_0)$ and $w:=w(s_0)$. Since $e_1(s_0)$ and $e_1(w)$ is past-directed timelike, we may write
    \begin{align}
        e_1(w) = \cosh(\phi)e_1(s_0)+\sinh(\phi)e_3 \label{eqs}
    \end{align}
    for some $g$-unit vector $e_3\perp e_1(s_0)$ and some hyperbolic angle $\phi\in [0,\infty)$. Denote $e_1:=e_1(s_0)$ and $Y:=Y(\rho(s_0),s_0)$.
    Using first $\dot c(s_0)\perp e_1-e_1(w)$ and then $\dot c(s_0)-Y\parallel e_1\perp e_3$ (Lemma \ref{ef}), we easily get
    \[
      \langle \dot c(s_0),e_1\rangle_g = -\coth\Big(\frac{\phi}{2}\Big)\, \langle \dot c(s_0),e_3\rangle_g  
      =-\coth\Big(\frac{\phi}{2}\Big)\, \langle Y,e_3\rangle_g.
    \]
    Hence, by Lemma \ref{ef},
    \begin{align}
        |v|_g\, \rho'(s_0) 
        &= -\coth\Big(\frac{\phi}{2}\Big)\, \langle Y,e_3\rangle_g -\langle Y,e_1\rangle_g
         \nonumber
        \\[5pt]
        (\text{by } \eqref{eqs})\qquad  &=
        -\coth\Big(\frac{\phi}{2}\Big)\, \Big\langle Y,\frac{e_1(w)}{\sinh(\phi)} 
        -\coth(\phi)\, e_1\Big\rangle_g 
        -\langle Y,e_1\rangle_g. \nonumber
    \\[5pt]
    &= \frac{1}{2\sinh^2(\phi/2)}\, \langle Y,e_1-e_1(w)\rangle_g. \label{eqae}
    \end{align}
    Up to this point, the computation was precise (compare \eqref{eqae} to (2.23) in \cite{Itoh/Tanaka}). Now we start estimating. Let $X=J(\cdot,v,w-v)$ denote the Jacobi field along $c_{v}$ with $X(0)=0$ and $\dot X(0)=w-v$. By Lemma \ref{lh}, we have 
    \[
        \Big|e_1(w)-e_1 + \frac{\dot X(\rho(s_0))}{|v|_g}\Big|_h \leq C_1\, \frac{|v-w|_h^2}{|v|_g}.
    \]
    Inserting this estimate into \eqref{eqae} and using \eqref{eqap}, we deduce
    \[
        |\rho'(s_0)|
        \leq
        \frac{1}{2\sinh^2(\phi/2)|v|_g^2}\,  \Big(|\langle Y,\dot X(\rho(s_0))\rangle_g|+C_1^2 \, |v-w|_h^2\, |Y|_h\Big).
    \]
    Being $Y(\cdot,s_0)$ and $X$ Jacobi fields along $c_v$ vanishing at $0$, we can swap the covariant derivative from $X$ to $Y(\cdot,s_0)$ (Lemma \ref{ld}); using then Lemma \ref{lh}(iii) together with \eqref{eqap}, we obtain
    \begin{align*}
        |\rho'(s_0)| 
        \leq 
        \frac{|v-w|_h^2}{2\sinh^2(\phi/2) |v|_g^2}\Big( C_1^2|\dot Y|_h+ C_1^2|Y|_h\Big).
    \end{align*}
    Here $\dot Y:=\dot Y(\rho(s_0),s_0)$. For the term in front of the paranthesis, we use Lemma \ref{lh}(ii), while we use \eqref{eqad}  and $Y(t,s_0)=J(t,v,v_1-v_0)$ to estimate the paranthesis. We end up with
    \begin{align}
        |\rho'(s_0)| 
        \leq 
        \frac{C_1^5|\dot c_{v}(\rho(s_0))-\dot c_{w}(\rho(s_0))|_h^2}{2\sinh^2(\phi/2) |v|_g^2}|v_0-v_1|_h
        = \frac{C_1^5 |e_1-e_1(w)|_h^2}{2\sinh^2(\phi/2)}|v_0-v_1|_h, \label{eqz}
    \end{align}
    where we used $|v|_g=|w|_g$ and the definition of $e_1$ and $e_1(w)$ in the last step.
    
    Recall that $E_1,...,E_n$ is an orthonormal basis for $h$ on $V$. Therefore, Lemma \ref{ll6} combined with the trivial bound $\sinh(\phi/2)\geq \phi/2$ (recall the well-known fact that $\phi$ is just the hyperbolic distance between $e_1$ and $e_1(w)$ on the unit hyperboloid in $T_{c_v(\rho(s_0))}M$) yields
    \[
            |\rho'(s_0)| \leq 2C_1^7\, \max\{|e_1|_h^2,|e_1(w)|_h^2\}\, |v_0-v_1|_h.
    \]
    Lemma \ref{lh}(i) combines with $|w|_g=|v|_g$ to assert $|e_1(w)|_h\leq C_1 |v|_g^{-1}$, the same bound holding also for $e_1$. Inserting this into the above inequality, and using the fact that, thanks to the concavity of $|\cdot|_g$ on the future causal cone $\C_x$, it holds $|v|_g\geq \min\{|v_0|_g,|v_1|_g\}$, we conclude the proof.
\end{proof}

\section{Lipschitz continuity in null directions}\label{sec: ext. to null}

In this short section, we prove Corollary \ref{cor3}.

\begin{proof}[Proof of Corollary \ref{cor3}]
    It suffices to prove Lemma \ref{la} without the denominator in \eqref{eqaf} 
    Indeed, once this is proven, Corollary \ref{cor3} follows mutatis mutandis as Theorem \ref{thmd} followed from Lemma \ref{la}. 

    We repeat the same proof as above, in particular we stick to the notation in the introduction of Section \ref{sa}. However, we will choose $U$ possibly smaller: First of all, we may suppose that for $v\in U$ and $w\in T_{x}M$ future directed causal with $|v|_g=|w|_g$ and $c_v(\rho(v))=c_w(\rho(w))$, we have
    \[
        \langle \dot c_v(\rho(v)),\dot c_w(\rho(w))\rangle_g \leq -\ep
    \]
    for some uniform $\ep>0$. This is an easy consequence of the fact that $c_{v_*}(\rho(v_*))$ is not conjugate to $x$ along $c_{v_*}$ (see Lemma \ref{lj}). Moreover, since $|v_*|_g=0$, we may suppose that $\cosh^{-1}(\ep/|v|_g^2)\geq 1$ for all future directed timelike $v\in U$.

    We do the same computations as in the preceding section up to \eqref{eqz}, which reads
   \begin{align}
        |\rho'(s_0)| \leq \frac{C_1^5\, |e_1-e_1(w)|_h^2}{2\sinh^2(\phi/2)}\, |v_0-v_1|_h.  \label{eqzz}
    \end{align}
    By \eqref{eqs}, the hyperbolic angle $\phi$ between $e_1$ and $e_1(w_0)$ is given by
    \[
        \phi = \cosh^{-1}(-\langle e_1,e_1(w)\rangle_g) \geq \cosh^{-1}\Big(\frac{\ep}{|v|_g^2}\Big)\geq 1.
    \]
    Hence, using $\sinh(a/2)^2\geq \cosh(a)/8$, $a\geq 1$, we get
    \[
        \sinh^2(\phi/2) \geq \frac{\ep}{8|v|_g^2}.
    \]
    Therefore, we conclude from \eqref{eqzz} that
    \[
        |\rho'(s_0)| \leq \frac{8C_1^5\, |\dot c_{v}(\rho(v))-\dot c_{w}(\rho(w))|_h^2}{\ep}\, |v_0-v_1|_h. 
    \]
    We conclude the proof in view of Lemma \ref{lh}(i).
\end{proof}

\appendix
\renewcommand{\thesection}{\Alph{section}}
\setcounter{section}{0}

\section*{Appendix} 

\section{Cut points and conjugate points}\label{App: A}

In this appendix, we present a brief survey of cut points and conjugate points. We begin by reviewing the definitions and basic properties of Jacobi fields and then establish the relationships between conjugate points, the Lorentzian index form, and the maximality of geodesics. We conclude by proving continuity properties of the focalization time and the cut time.

Most of the results presented here are classical and can be found in \cite{Ehrlich} (see also \cite{ONeill}). However, since these results are scattered across several chapters of the literature, and a few of the statements included here are new, we provide a unified and almost self-contained exposition.

Recall that $(M,g)$ always denotes a globally hyperbolic spacetime.

\subsection{Conjugate points and the focalization time}

\begin{definition}[Jacobi fields and conjugate points]\rm\label{defa}
Let $c:I\to M$ be a geodesic.
    \begin{enumerate}[(a)]
        \item A \emph{Jacobi field along $c$} is a smooth vector field $J$ along $\gamma$ such that
        \begin{align}
            \ddot J+R(J,\dot c)\dot c=0, \label{eqar} 
        \end{align}
        where $\ddot J=\frac{D^2J}{dt^2}$ denotes the second covariant derivative and 
        \[
        R(X,Y)Z:=\nabla_X \nabla_Y Z-\nabla_Y \nabla_X Z-\nabla_{[X,Y]}Z
        \]
        is (our convention of) the Riemannian curvature tensor.

    Writing \eqref{eqar} in coordinates w.r.t.\ an orthonormal frame along $c$, the equation turns out to be a linear ODE (\cite{DoCarmo}, Chapter 5; see also Lemma \ref{lc}(a)). It follows that Jacobi fields are uniquely determinded by the initial conditions $J(0)$ and $\dot J(0)$. If $x\in M$ is fixed, we denote by $J(t,v,w)$. $v,w\in T_xM$, the unique Jacobi field along $c_v$ with $J(0,v,w)=0$ and $\dot J(0,v,w)=w$. Recall that $c_v$ is the geodesic emerging from $x$ with initial velocity $v$.
        \item 
        Let $t_0,t_1\in I$, $t_0\neq t_1$. Then $c(t_1)$ is said to be a \emph{conjugate point} of $c(t_0)$ along $c$ if there exists a non-zero Jacobi field $J$ along $c$ such that $J(t_0)=J(t_1)=0$. 
    \end{enumerate}
\end{definition}

\begin{remark}\rm \label{hdsdoa}
    Let $c:[0,a]\to M$ be a geodesic. Then $c(a)$ is a conjugate point of $c(0)$ along $c$ if and only if $a\dot c(0)$ is a critical point of the exponential map $\exp_{c(0)}:T_{c(0)}M\to M$, i.e.\ if
    \[
        \op{ker}(d_{a\dot c(0)}(\exp_{c(0)}))\neq \{0\}
    \]
    (Proposition 3.5 in \cite{DoCarmo}).
    This follows from the fact that the Jacobi field $J(t,v,w)$ is explicitely given by
    \begin{align}
        J(t,v,w)=d_{t\dot c(0)}\exp_{c(0)}(tw). \label{eq: Jac. repr.}
    \end{align}
    Since $d_0\exp_x=\op{Id}_{T_xM}$, it follows that each $v_0\in T_xM$ admits an open neighbourhood $U$ and some $\ep>0$ such that $J(t,v,w)\neq 0$ for all $v\in U$, $0\neq w\in T_xM$ and $t\in [0,\ep]$.
\end{remark}

\begin{lemma}\label{ld}
    Let $J,X$ be Jacobi fields along a common geodesic. Then $\langle \dot J,X\rangle_g -\langle J,\dot X\rangle_g$ is constant.
\end{lemma}
\begin{proof}
    Follows by differentiating and using the symmetries of $R$.
\end{proof}

\begin{lemma}\label{lc}
    \begin{enumerate}[(i)]
        \item Let $x\in M$. Then $J=J(t,v,w)$ depends smoothly on $t$, $v$ and $w$.
        \item Let $V\subseteq M$ be a convex set. For $a>0$ and $x,y\in V$, let $c:[0,a]\to M$ be the unique geodesic in $V$ with $c(0)=x$ and $c(a)=y$. For any $u\in T_yM$, there exists a unique Jacobi field $J$ along $c$ such that $J(0)=0$ and $J(a)=w$. Moreover, $J$ depends smoothly on $a$, $x$, $y$ and $w$.
        \end{enumerate}
\end{lemma}
\begin{proof}
     For (i), let $E_i$, $i=1,...,n$, be an orthonormal basis of $T_xM$, with $E_1$ future directed timelike. Denote by $E_i(t,v)$ the parallel transport of $E_i$ along $c_v$, and note that $E_i$ depends smoothly on $t$ and $v$.
    
    For each Jacobi field $J(t,v,w)$, we can write (as in \cite{DoCarmo}, Chapter 5)
    \[
        J(t,v,w)=\sum_{i=1}^n f_i(t,v,w) E_i(t,v)
    \]
    for smooth functions $f_i(\cdot,v,w)$. The Jacobi equation becomes
    \[
        \left\{
        \begin{aligned}
            &\ep_i\, \ddot f_i(t,v,w) 
            + \sum_{j=1}^n f_j(t,v,w)\, 
              \bigl\langle R(E_j(t,v,w),\dot c_v)\dot c_v, E_i(t,v,w) \bigr\rangle_g = 0,\\[2mm]
            &f_i(0,v,w) = \ep_i\, \langle v, E_i \rangle_g, \quad 
             \dot f_i(0,v,w) = \ep_i\, \langle w, E_i \rangle_g, 
             \quad i = 1, \dots, n,
        \end{aligned}
        \right.
    \]
    with $\ep_i=-1$ for $i=1$ and $\ep_i=1$ otherwise.
    The conclusion follows since the solution of the linear ODE depends smoothly on $t$, $v$ and $w$.

    For (ii), let $v:=a^{-1}\exp_x^{-1}(y)$, so that $c_v(a)=y$ and hence $c=(c_v)|_{[0,a]}$. Since $V$ is convex, $d_{\exp_x^{-1}(y)}\exp_x$ is an isomorphism. Hence, the representation \eqref{eq: Jac. repr.} implies that $J$ is given by
    \[
        J(t):=J(t,v,d_{\exp_y^{-1}(x)}\exp_y(w)/a),
    \]
    which depends smoothly on $a$, $x$, $y$ and $w$ by convexity of $V$.
\end{proof}

\begin{definition}[Lorentzian index form, \cite{Ehrlich} Definition 10.4]\rm \label{synge}
    Let $c:[a,b]\to M$ be a future directed causal\footnote{Our definition differs from \cite{Ehrlich} as we allow for null geodesics.} geodesic, and denote by $V^\perp(c)$ the set of all piecewise smooth vector fields $X$ along $c$ such that $\langle X,\dot c\rangle_g\equiv0$ and $X(a)=X(b)=0$. The \emph{Lorentzian index form} is the symmetric and bilinear form $I:V^\perp(c)\times V^\perp(c)\to \R$ defined by
    \begin{align}
        I(X,Y)&:=-\int_a^b \langle \dot X,\dot Y\rangle_g -\langle R(X,\dot c)\dot c,Y\rangle_g\, dt \label{eqh}
        \\[5pt]
        & =\sum_{i=1}^n \Big\langle \dot X(a_i^+)-\dot X(a_i^-),Y(a_i)\Big\rangle_g
        + \int_a^b \langle \ddot X+R(X,\dot c)\dot c,Y\rangle_g\, dt, \nonumber
    \end{align}
    where $a=:a_0<a_1<...<a_n=:b$ is a partition of $[a,b]$ such that $X$ and $Y$ are smooth on each $[a_i,a_{i+1}]$, $R$ is the curvature tensor, and 
    \[
        \dot X(a_i^+) = \lim_{t\downarrow a_i} \dot X(t),\quad 
        \dot X(a_i^-) = \lim_{t\uparrow a_i} \dot X(t).
    \]
\end{definition}

\begin{lemma}[Synge's second variation formula]
     Let $c_s:[a,b]\to M$ be a smooth and proper\footnote{I.e.\ $c_s(a)=c(a)$ and $c_s(b)=c(b)$ for all $s$.} variation of the future directed timelike geodesic $c:[a,b]\to M$ with variational vector field $X(t):=\frac{\partial}{\partial s}\Big|_{s=0}c_s(t)$ belonging to $V^\perp(c)$. Then
    \begin{align*}
        \frac{d}{ds}\Big|_{s=0} \ell_g(c_s)=0\quad \text{ and }\quad 
        \frac{d^2}{ds^2}\Big|_{s=0} \ell_g(c_s)=\frac{1}{|\dot c|_g} I(X,X).
    \end{align*}
\end{lemma}
\begin{proof}
    Since $c$ is timelike, we can differentiate under the integral sign and obtain
    \[
        \frac{d}{ds} \ell_g(c_s)= \int_a^b \frac{-\langle \dot c_s(t),\frac{D}{\partial s}\dot c_s(t)\rangle_g}{|\dot c_s(t)|_g}\, dt
        =
        \int_a^b \frac{-\langle \dot c_s(t),\frac{D}{\partial t}\frac{\partial}{\partial s} c_s(t)\rangle_g}{|\dot c_s(t)|_g}\, dt.
    \]
    For $s=0$, integrating by parts, noticing $|\dot c|_g$ is constant, that $c$ is geodesic and that $X(a)=X(b)=0$, we get the first part of the lemma. 
    Moreover, 
    \begin{align*}
        \frac{d^2}{ds^2}\Big|_{s=0} \ell_g(c_s)
        =
        \int_a^b \frac{-\big\langle \dot c(t),\dot X(t)\big\rangle_g^2}{|\dot c(t)|^3_g}\,  
        -
        \frac{\langle \dot X(t),\dot X(t)\rangle_g}{|\dot c(t)|_g}
        - 
        \frac{\Big\langle \dot c(t),\frac{D}{\partial s}\Big|_{s=0}\frac{D}{\partial t}\frac{\partial}{\partial s} c_s(t)\Big\rangle_g}{|\dot c(t)|_g}\, dt.
    \end{align*}
    The integrand in the first integral is $0$, as follows easily by definition of $V^\perp(c)$. Moreover,
    \[
        \frac{D}{\partial s}\Big|_{s=0}\frac{D}{\partial t}\frac{\partial}{\partial s} c_s(t)
        =
        \frac{D}{\partial t}\frac{D}{\partial s}\Big|_{s=0}\frac{\partial}{\partial s} c_s(t)+R\Big(\frac{\partial c_s}{\partial s}\Big|_{s=0}(t),\frac{\partial c}{\partial t}(t)\Big)\frac{\partial c_s}{\partial s}\Big|_{s=0}(t).
    \]
    The scalar product of $\dot c$ with the first term on the right-hand side integrates to zero since $c$ is a geodesic and $c_s$ is a  proper variation. Thus, using the symmetries of the curvature tensor, we obtain
    \begin{align*}
        \frac{d^2}{ds^2}\Big|_{s=0} \ell_g(c_s)
        =
        -\frac {1}{|\dot c|_g}
        \int_a^b \langle \dot X(t),\dot X(t)\rangle_g
        - \langle R(X,\dot c)\dot c,X\rangle_g\, dt.
    \end{align*}
    This proves the lemma.
\end{proof}

For a maximizing geodesic, it follows that the Lorentzian index form must be negative semidefinite:

\begin{corollary}[Lorentzian index form and maximality]\label{proc}
    Let $c:[a,b]\to M$ be a future directed timelike geodesic, and suppose that the Lorentzian index form is not negative semidefinite. Then $c$ is not maximizing.
\end{corollary}

\begin{proof}
    If $I(X,X)>0$, it suffices to pick a piecewise smooth and proper variation $c_s$ of $c$ with variational vector field $X$.
\end{proof}
The converse result is not true; think of the Lorentzian cylinder $\R\times S^1$.

We now investigate the relationship between conjugate points and the Lorentzian index form. We first show that if $I$ is not negative semidefinite, then there exists a conjugate point along $c$.

\begin{proposition}[Lorentzian index form and conjugate points I, \cite{Ehrlich} Proposition 10.21 and Theorem 10.22] \label{prodd}
    Let $c:[a,b]\to M$ be a future directed timelike geodesic, and suppose that $I$ is not negative semidefinite. Then there is $t\in (a,b)$ such that $c(t)$ is conjugate to $c(a)$ along $c$.
\end{proposition}
\begin{proof}
    Let $I(X,X)>0$ for some $X\in V^\perp(c)$ and let $c_s$ be a smooth proper variation of $c$ with variational vector field $X$. Suppose $c$ had no conjugate point. To obtain a contradiction, in view of Lemma \ref{synge}, it suffices to show that $\ell_g(c_s)\leq \ell_g(c)$.
    
    Since $c$ has non conjugate points, $\exp_{c(a)}$ is a local diffeomorphism around $t\dot c(a)$ for every $t\in [0,b-a]$ by \eqref{eq: Jac. repr.}. Covering the line $(t-a)\dot c(a)$, $t\in [a,b]$, by finitely many open neighbourhoods which are mapped diffeomorphically to a neighbourhood of $c(t)$ via $\exp_{c(a)}$, it is easy to construct a smooth family of curves $\gamma_{s}(t)$ in $T_{c(a)}M$ very close to $(t-a)\dot c(a)$ such that $c_s(t)=\exp_{c(a)}(\gamma_s(t))$ (for $s$ sufficiently small). Since $\exp_{c(a)}$ is a diffeomorphism around $(b-a)\dot c(a)$, we must have $\gamma_s(b)=(b-a)\dot c(a)$. It then follows from the Gauss lemma that $\ell_g(c_s)\leq \ell_g(c)$ for sufficiently small $s$: 
    
    Indeed, observe first that the timelike character of $c$ guarantees that $c_s$ is future directed timelike for small $s$. Since $c_s(t)=\exp_{c(a)}(\gamma_s(t))$, it is well-known that $\gamma_s(t)=\exp_{c(a)}^{-1}(c_s(t))\in \op{int}(\C_{c(a)})$ for all $t\in (a,b]$ (\cite{Baer}, Lemma 2.11). Consequently, we may write
    \[
        \gamma_s(t) = r_s(t) v_s(t),\quad t\in (a,b],
    \]
    where $r_s(t)\in (0,\infty)$ and $v_s(t)\in \op{int}(\C_{c(a)})$ with $|v_s(t)|_g=1$ are two piecewise smooth curves (this idea is taken from \cite{DoCarmo}, Chapter 3, Proposition 3.6). Setting $f_s(r,t):= \exp_{c(a)}(r v_s(t))$, it follows that $c_s(t)= f_s(r_s(t),t)$, and thus
    \[
        \dot c_s(t) = \frac{\partial f_s}{\partial r}(r_s(t),t) \dot r_s(t) +\frac{\partial f_s}{\partial t}(r_s(t),t).
    \]
    The Gauss lemma (\cite{Baer}, Proposition 2.9) implies for all $r$ and $t$
    \begin{align*}
        g\Big(\frac{\partial f_s}{\partial r},\frac{\partial f_s}{\partial t}\Big) 
        &= g\big(d_{rv_s(t)}\exp_{c(a)}(v_s(t)),d_{rv_s(t)}\exp_{c(a)}(r\dot v_s(t))\big) 
        \\[5pt]
        &= g(v_s(t), r\dot v_s(t)) 
        \\[5pt]
        &=  \frac r2 \frac{d}{dt} g(v_s(t),v_s(t))=0.
    \end{align*}
    Furthermore, $g(\frac{\partial f_s}{\partial r},\frac{\partial f_s}{\partial r})=-1$, while $g(\frac{\partial f_s}{\partial t},\frac{\partial f_s}{\partial t})\geq 0$ since $\frac{\partial f_s}{\partial t}$ is orthogonal to the timelike vector $\frac{\partial f_s}{\partial r}$ and therefore spacelike (or zero). Thus
    \[
        g(\dot c_s(t),\dot c_s(t)) \geq -\dot r_s(t)^2 \quad \forall t\in (a,b]
    \]
    Assuming for the moment that $\dot r_s(t)\ge0$, we obtain
    \[
        \ell_g(c_s) = \int_a^b |\dot c_s(t)|_g \, dt \leq \int_a^b \dot r_s(t)\, dt = r_s(b) = (b-a) |\dot c(a)|_g = \ell_g(c),
    \]
    which would conclude the proof. It therefore remains to show that $\dot r_s(t)\ge0$. To this end, another application of the Gauss lemma yields
    \begin{align*}
         g(\dot \gamma_s(t),\gamma_s(t))&= g(d_{\gamma_s(t)}\exp_{c(a)}(\dot \gamma_s(t)),d_{\gamma_s(t)}\exp_{c(a)}(\gamma_s(t)))
         \\[5pt]
         &=g(\dot c_s(t),d_{\gamma_s(t)}\exp_{c(a)}(\gamma_s(t))).
    \end{align*}
    The right-hand side is strictly negative because both $\dot c_s(t)$ and $d_{\gamma_s(t)}\exp_{c(a)}(\gamma_s(t))$ are future-directed timelike vectors. On the other hand,
    \[
        0 > g(\dot \gamma_s(t),\gamma_s(t)) = -\dot r_s(t)r_s(t),
    \]
    implying $\dot r_s(t)>0$ as required.
\end{proof}

The following proposition states the converse result. Note that the first part also allows for causal geodesics.

\begin{proposition}[Lorentzian index form and conjugate points II, \cite{Ehrlich} Proposition 10.12]\label{prod}
    Let $c:[a,b]\to M$ be a future directed causal geodesic that contains a conjugate point in $(a,b)$. Then $I$ is not negative semidefinite, i.e.\ there exists $X\in V^\perp(c)$ with $I(X,X)>0$. In particular, if $c$ is timelike, then $c$ is not maximizing.
\end{proposition}

\begin{proof}
    Our proof follows Theorem 10.12 in \cite{Ehrlich} (see also \cite{Klingenberg}), with suitable modifications to treat the null case.
    
    Let $t_0\in (a,b)$ be such that $c(t_0)$ is conjugate to $c(a)$ along $c$, let $J$ be a corresponding Jacobi field, and let $\psi:[a,b]\to [0,1]$ be smooth with $\psi=0$ near $a$ and $b$ and $\psi=1$ near $t_0$. Define $\tilde X$ to be the parallel vector field along $c$ with $\tilde X(t_0)=-\dot J(t_0)$, and set $X:=\psi\tilde X$. Finally, for $\ep>0$, define 
    \begin{align*}
        X_\ep:[a,b]\to TM,\ X_\ep(t):=
        \begin{cases}
            J(t)+\ep X(t),\ &\text{if } a\leq t\leq t_0,
            \\
            \ep X(t),\ &\text{if } t_0\leq t\leq b.
        \end{cases}
    \end{align*}
    Obviously, $X_\ep$ is piecewise smooth with $X_\ep(a)=X_\ep(b)=0$. Moreover, since $J\in V^\perp(c)$, we have $\langle J, \dot c\rangle_g\equiv 0$ as well as $\langle \dot J,\dot c\rangle_g\equiv 0$. Since $\tilde X$ is a parallel along $c$ with $\tilde X(t_0)=-\dot J(t_0)$, it follows that also $\langle \tilde X, \dot c\rangle_g\equiv 0$, hence $X_\ep\in V^\perp(c)$. 

    We then compute
    \begin{align*}
        I(X_\ep,X_\ep) 
        &= \langle -\dot J(t_0),-\ep \dot J(t_0)\rangle_g 
        +
        \int_a^{t_0} \langle \ep \ddot X+\ep R(X,\dot c)\dot c,J+\ep X\rangle_g\, dt+O(\ep^2)
        \\[5pt]
        & =\ep \langle \dot J(t_0),\dot J(t_0)\rangle_g +\ep \int_a^{t_0} \langle \ddot X+ R(X,\dot c)\dot c,J\rangle_g\, dt+O(\ep^2).
    \end{align*}
    Integrating by parts and using that $J$ is a Jacobi field vanishing at $a$ and $t_0$, we get that the second term is equal to $\ep \langle \dot J(t_0),\dot J(t_0)\rangle_g$, so we end up with
    \[
        I(X_\ep,X_\ep) 
        =
        2\ep \langle \dot J(t_0),\dot J(t_0)\rangle_g+O(\ep^2).
    \]
    Now note that $\dot J(t_0)$ is orthogonal to $\dot c(t_0)$. However, if $\dot J(t_0)$ was a multiple of $\dot c(t_0)$, then $J(t)$ would be equal to $\mu(t-t_0)\dot c(t)$ ($\mu\in \R$), and, in particular, $J\equiv 0$ (if $\mu=0$) or $J(a)\neq 0$ (if $\mu\neq 0$); both is impossible. Thus, $\dot J(t_0)$ is orthogonal but not a multiple of $\dot c(t_0)$. Hence, $\dot J(t_0)$ is spacelike and thus $\langle \dot J(t_0),\dot J(t_0)\rangle_g>0$.
    It is then obvious that $I(X_\ep,X_\ep)>0$ for small $\ep$.

    The fact that $c$, if timelike, is not maximizing, follows from Corollary \ref{proc}.
\end{proof}

The following theorem extends the second part of Proposition \ref{prod} to future directed causal geodesics. The proof is more difficult and the computations become more tedious. We omit it here, but refer to \cite{Baer}. Note that the direct argument using the Lorentzian index form fails since $c$ is not longer assumed to be timelike.

\begin{theorem}[Conjugate points and maximality for null geodesics]\label{thmh}
    Let $c:[a,b]\to M$ be a future directed causal geodesic that contains a conjugate point in $(a,b)$. Then $c$ is not maximizing.
\end{theorem}

We conclude with the continuity of the focalization time. We recall the definition.

\begin{definition}[Focalization time]\rm
  Fix $x\in M$. The \emph{focalization time} $\lambda:\C_x\to (0,\infty]$ is defined as
    \[
        \lambda(v):=\inf\{t\geq 0\mid c_v(t) \text{ is conjugate to $x$ along } c_v\}
    \]
    We only define $\lambda$ on $\C_x$, although a definition on $T_xM$ is possible too.  
\end{definition}

Note that, by Remark \ref{hdsdoa}, due to the smoothness of the exponential function, $\lambda$ is lower semicontinuous at $v\in \C_x$ unless $\liminf_{\C_x\ni w\to v} \lambda(w)$ is finite but $c_v$ is not defined up to that point.

\begin{lemma}[Continuity of the focalization time]\label{cor2}
    $\lambda$ is everywhere upper semicontinuous. Moreover, $\lambda$ is continuous at $v$ unless $\liminf_{\C_x\ni w\to v} \lambda(w)$ is finite but $c_v$ is not defined up to that point. In particular, $\lambda$ is continuous on its domain $\lambda^{-1}(\R)$, which is open.
\end{lemma}
\begin{proof}[Proof of Lemma \ref{cor2}]
    Fix $v_0\in T_xM$ future directed causal. Recall from Remark \ref{hdsdoa} that we only have to show upper semicontinuity. If $\lambda(v_0)$ is infinite, then upper semicontinuity is trivial. Otherwise, we assume that $\lambda(v_0)<\infty$. 
    
    Let $\ep>0$ be given; we have to show $\lambda(v)\leq \lambda(v_0)+\ep$ for all $v$ future directed causal and sufficiently close to $v_0$. We first claim that it is sufficient to consider the case where $v$ is timelike. This is trivial if $v_0$ is timelike as the set of future directed timelike vectors is open. Otherwise, let $v_0$ be null. Then, by assumption, there exists a neighbourhood $U$ of $v_0$ such that, for all $v\in U$ future directed timelike, we have $\lambda(v)\leq \lambda(v_0)+\ep$. We can also assume that $c_v$ is defined up to $\lambda(v_0)+\ep$ for all $v\in U$. If $v\in U$ is null, then it is the limit of a sequence $v_k\in U$ of timelike vectors, and  since $c_v$ is defined up to $\liminf_{k\to \infty} \lambda(v_k)\leq \lambda(v_0)+\ep$, it follows from the smoothness of the exponential map (resp.\ the lower semicontinuity at $v$) that 
    \[
        \lambda(v)\leq \liminf_{k\to \infty} \lambda(v_k)\leq \lambda(v_0)+\ep,
    \]
    which yields the claim.
    
    We now prove the case where $v$ is timelike. Fix any $\ep>0$ such that $c_v$ is defined on $[0,\lambda(v_0)+\ep]$ for all $v$ close to $v_0$, and denote by $I_v$ the Lorentzian index form associated to $c_{v}$ restricted to $[0,\lambda(v_0)+\ep]$. Denote by $E_i$ a orthonormal basis of $T_{x}M$ with $E_1$ future directed timelike. For $v$ close to $v_0$, let $E_i(t,v)$, $t\in [0,\lambda(v_0)+\ep]$ be the parallel transport of $E_i$ along $c_v$. 
    
    By Proposition \ref{prod}, there exists $X_0\in V^\perp(c_{v_0})$ such that $I_{v_0}(X_0,X_0)>0$. For some piecewise smooth $a_i(t)$, we may write
    \[
        X_0(t) = \sum_{i=1}^n a_i(t) E_i(t,v_0).
    \]
    For $v$ close to $v_0$, define the vector field
    \[
        X(t) = \sum_{i=1}^n b_i(t) E_i(t,v) 
    \]
    with $b_i=a_i$ for $i\geq 2$ and $b_1$ such that $X\in V^\perp(c_v)$, i.e.\ 
    \[
        b_1(t) = -\frac{1}{\langle \dot c_v(t),E_1(t,v)\rangle_g}\sum_{i=2}^n a_i(t) \langle E_i(t,v),\dot c_v(t)\rangle_g.
    \]
    Note that $X(0)=X(\lambda(v_0)+\ep)=0$ since $a_i(0)=a_i(\lambda(v_0+\ep))=0$. Also, $b_1$ converges to $a_1$ in the $C^1$-topology (on each subinterval where $X_0$ is smooth) as $v\to v_0$. It follows that $I_v(X,X)>0$ when $v$ is sufficiently close to $v_0$. Since $v$ is timelike, Proposition \ref{prodd} yields $\lambda(v)<\lambda(v_0)+\ep$, and this proves the upper semicontinuity.
\end{proof}

\subsection{Cut points and the cut time}

\begin{definition}[Cut time and cut locus]\rm
    Fix $x\in M$. The \emph{cut time} is the function $\rho:\C_x\to (0,\infty]$ defined by
    \begin{align*}
    \rho(v):=\sup\{t\geq 0\mid d(x,c_v(t))=t|v|_g\}.
    \end{align*}
    If $c_v$ only exists up to some $t>0$, then $\rho(v)\leq t$.
    The point $c_v(\rho(v))$ – if it exists – is called a \emph{cut point of} $x$. The set of cut points is denoted by $\op{Cut}(x)$.
\end{definition}

\begin{lemma}[Compactness of maximizing geodesics]\label{lm}
    Let $x\in M$ and $v_k\in \C_{x}$ such that the curves $[0,1]\ni t\mapsto c_{v_k}(t)$ are maximizing geodesics connecting $x$ to $y_k:=c_{v_k}(1)$. Suppose that $y_k\to y$. Then $(x,y)\in J^+$, and $v_k$ converges, along a subsequence, to some $v\in \C_x$. Moreover, $[0,1]\ni t\mapsto c_v(t)$ is well-defined and a maximizing geodesic connecting $x$ to $y$.
\end{lemma}
\begin{proof}
    Since $J^+$ is closed, $(x,y)\in J^+$. The case $x\ll y$ follows from Lemma 9.6 in \cite{Ehrlich}, together with the continuity of $d$ and a scaling argument. The case $x=y$ is trivial: for large $k$, strong causality implies $v_k=\exp_{x}^{-1}(y_k))\xrightarrow{k\to \infty} 0$. If $x\neq y$ and $d(x,y)=0$, the claim follows from Lemma 9.14 and the proof of Lemma 9.25 in \cite{Ehrlich}.
\end{proof}

\begin{theorem}[Characterization of cut points]\label{thme}
    Let $c:[0,a)\to M$ be a causal geodesic emerging from $x$. If $y=c(t_0)\in J^+(x)$ is the cut point of $x$ along $c$, then at least one of the following holds:
    \begin{enumerate}[(i)]
        \item $y$ is the first conjugate point of $x$ along $c$.
        \item There exists a second distinct maximizing geodesic connecting $x$ to $y$.
    \end{enumerate}
    Conversely, suppose $c$ is maximizing on ${[0,t_0]}$. If either $y=c(t_0)$ is a conjugate point of $x$ along $c$ or (ii) holds, then $c$ is not maximizing beyond $t=t_0$.
\end{theorem}

\begin{proof}[Proof of Theorem \ref{thme}]
    The implication $\Rightarrow$ is an easy consequence of the compactness result in Lemma \ref{lm} applied to $y_k=c(t_0+1/k)$ (see Theorems 9.12 and 9.15 in \cite{Ehrlich}). Conversely, suppose that $c$ is maximizing on $[0,t_0]$. The fact that (ii) implies $\gamma$ is not maximizing beyond $t_0$ follows from the fact that any maximizing curve must be a pregeodesic \cite{Minguzzi}, Theorem 2.9 (if $c$ is maximizing on $[0,t_0+\ep]$, concatenate $c_{|[t_0,t_0+\ep]}$ with a different maximizing geodesic from $x$ to $y$ to obtain a maximizing curve that is no pregeodesic).
    The fact that $\gamma$ is not maximizing beyond a conjugate point follows Theorem \ref{thmh}, whose prove we only gave in the timelike case.
\end{proof}

\begin{lemma}[Continuity of the cut time]\label{lq}
    $\rho$ is lower semicontiniuous. Moreover, $\rho$ is continuous at $v$ unless $\rho(v)$ is finite and $c_v(\rho(v))$ is not defined. 
\end{lemma}

Note that the roles of upper and lower semicontinuity are reversed compared to Lemma \ref{cor2}.

\begin{proof}[Proof of Lemma \ref{lq}]
    (See \cite{Ehrlich}, Proposition 9.33.) If $\rho(v)$ is finite and $c_v(\rho(v))$ is defined, then upper semicontinuity at $v$ follows easily from the continuity of the time separation function. To show lower semicontinuity, suppose first $\rho(v)<\infty$, and assume that $v_k\to v$ but $\rho(v_k)\leq \rho(v)-2\ep$ for some $\ep>0$ and all $k\in \N$. Then there exist maximizing geodesics from $x$ to $c_{v_k}(\rho(v)-\ep)$, say $c_{w_k}(t)$, $t\in [0,\rho(v)-\ep]$. By Lemma \ref{lm}, $w_k$ converges, along a subsequence, to some $w$, and $c_w(t)$, $t\in [0,\rho(v)-\ep]$, is a maximizing geodesic from $x$ to $c_v(\rho(v)-\ep)$. Necessarily, $w=v$, otherwise there would exist two distinct maximizing geodesics from $x$ to $c_v(\rho(v)-\ep)$ and $c_v$ would fail to be maximizing beyond. Hence,
    \[
        v_k,w_k\xrightarrow{k\to \infty} v \text{ and } c_{v_k}(\rho(v)-\ep), c_{w_k}(\rho(v)-\ep)\xrightarrow{k\to \infty} c_v(\rho(v)-\ep).
    \]
    Hence, $d_{(\rho(v)-\ep)v}\exp_x$ must be singular, so that $c_v$ contains a conjugate point in $[0,\rho(v))$ and cannot be maximizing. This is a contradiction. 

    The case $\rho(v)=\infty$ is very similar. Instead of $\rho(v)-2\ep$, choose $T>0$ such that $\liminf_{k\to \infty} \rho(v_k)\leq T$, and proceed with the argument applied to $T+\ep$ instead of $\rho(v)-\ep$.
\end{proof}

\begin{lemma}\label{lr}
    Let $x\in M$ and $K\subseteq \C_x$ be a compact subset such that $c_v(\rho(v))$ is defined for all $v\in K$. Then the set 
    \[
        \{(v,w)\in K\times \C_x\mid \rho(w)=\rho(v) \text{ and } c_w(\rho(w))=c_v(\rho(v))\}
    \]
    is compact. 
\end{lemma}
\begin{proof}
    Let $v_k \in U$ and $w_k \in \mathcal{C}_x$ be sequences as in the hypothesis of the lemma.
    Passing to a subsequence, we may assume that $v_k \to v \in K$. Thanks to our assumption, $\rho$ is continuous at $v$ so that $\rho(v_k)\to \rho(v)$. A simple scaling argument together with Lemma \ref{lm} grants that $c_v:[0,\rho(v)]\to M$ is well-defined and a maximizing geodesic. Applying Lemma \ref{lm} again (but now with $y_k = c_{w_k}(\rho(w_k))$ and $y = c_v(\rho(v))$, noticing that $\rho(w_k)=\rho(v_k)\to \rho(v)$), we obtain – after possibly passing to a further subsequence – that 
    $w_k \to w$ for some $w\in \C_x$ and that $c_w:[0,\rho(v)]\to M$ is a maximizing geodesic from $x$ to $y$. In particular, $\rho(w)\geq \rho(v)$, but on the other hand the lower semicontinuity of $\rho$ gives 
    \[
        \rho(w) \leq \liminf_{k\to \infty} \rho(w_k) = \rho(v).
    \] 
    To summarize, up to a subsequence, $(v_k,w_k)\to (v,w)$ and $v,w$ satisfy the hypothesis of the lemma. This proves the claim.
\end{proof}

\begin{lemma}[Non-conjugacy implies uniform linear independency]\label{lj}
   Let $x$ and $v_0\in \C_x$ be such that $\rho(v_0)<\lambda(v_0)$ and such that $c_{v_0}(\rho(v_0))$ is defined. Then there exists $\ep>0$ such that the following holds: If $v$ future directed causal is sufficiently close to $v_0$, then for all $w\neq v$ future directed causal with $\rho(v)=\rho(w)$ and $c_v(\rho(v))=c_w(\rho(w))$, we have
   \[
        \bigl\langle\dot c_v(\rho(v)),\dot c_w(\rho(w))\bigl\rangle_g\, 
        \leq -|v|_g|w|_g-\ep.
   \]
   Moreover, there exists at least one such $w$ (if $v$ is close to $v_0$).
\end{lemma}
\begin{proof}   
   For $v$ sufficiently close to $v_0$, Lemmas \ref{cor2} and \ref{lq} give $\rho(v)<\lambda(v)$; thus, Theorem \ref{thme} yields the existence of a second (i.e.\ not reparametrized) maximizing geodesic $c_w:[0,\rho(w)]\to M$ from $x$ to $c_v(\rho(v))$. Up to rescaling, we may assume $\rho(w)=\rho(v)$. This proves the the second statement.
   
    To prove the first part, note first that the claimed inequality always holds for $\ep=0$ by the reversed Cauchy Schwarz inequality. Now, suppose by contradiction that we find sequence $v_k\in \C_{x}\cap U$ converging to $v_0$ and $w_k$ as in the hypothesis such that
    \begin{align}
         \big\langle\dot c_{v_k}(\rho(v_k)),\dot c_{w_k}(\rho(w_k))\big\rangle_g +|v_k|_g|w_k|_g
         \xrightarrow{k\to \infty} 0. \label{eqa}
    \end{align}
    The preceding lemma gives $w_k\to w_0\in \C_x$ with $\rho(w_0)=\rho(v_0)$ and such that $c_{w_0}(\rho(w_0)))=c_{v_0}(\rho(v_0))$. Note that $w_0\neq v_0$ because $w_k\neq v_k$ and $c_{v_0}(\rho(v_0))$ is not conjugate to $x$ along $c_{v_0}$. Combining this with $\rho(w_0)=\rho(v_0)$ forces $\dot c_{w_0}(\rho(w_0))$ and $\dot c_{v_0}(\rho(v_0))$ to be linearly independent. Thus, the reverse Cauchy Schwarz inequality yields
    \begin{align}
        \big\langle g(\dot c_{v_0}(\rho(v_0)),\dot c_{w_0}(\rho(v_0))\big\rangle_g +|v_0|_g|w_0|_g<0.\label{eqp}
    \end{align}
    However, $\dot c_{v_k}(\rho(v_k))$ converges to $\dot c_{v_0}(\rho(v_0))$, and the same holds for $w_k$ and $w_0$. 
    In view of \eqref{eqa} and \eqref{eqp}, this gives a contradiction.
\end{proof}

\section{Proof of Proposition \ref{li}} \label{App: B}

The following lemma is useful for the proof of Proposition \ref{li}. Its proof is
a tedious computation.

\begin{lemma}[Semiconcavity criterion]
Let $f,g : U \to (0,\infty)$ be a smooth resp.\ semiconvex Lipschitz function defined on an open set $U \subset \mathbb{R}^n$. Suppose that $f$ is $1$-homogeneous with $f \le g$, and suppose $\rho : U'\subseteq U\to \mathbb{R}$ is a continuous function with $\rho(x)x \in U$ for $x \in U'$ and $f(tx)=g(tx)$
for all $x \in U'$ and $t \in [0,\rho(x)]$ such that $tx\in U$. Suppose that, for each $x\in U'$, there exists $y\in\partial^- g(\rho(x)x)$
with
\begin{align}
\langle y,x\rangle\geq (1+\ep)f(x). \label{eqcb}
\end{align}
Then $\rho$ is locally semiconcave on $U'$.
\end{lemma}

\begin{proof}
There exists a constant $C_1$ such that, for all $x\in U'$, we find
$y\in\mathbb{R}^n$ as in \eqref{eqcb} with
\[
g(x') \ge g(\rho(x)x)
 + \langle y,x'-\rho(x)x\rangle
 - C_1|x'-\rho(x)x|^2
 \quad \forall x'\in U.
\]
We may also assume that $|y|\le C_1$ since $g$ is Lipschitz.

Fix $x_0\in U'$, and let us show that $\rho$ is semiconcave in a neighbourhood of $x_0$. Let $r,s_0>0$ such that
\[
(\rho(x)+s)x' \in U
\]
for any $s\in[-s_0,s_0]$ and
$x,x'\in B_r(x_0)\Subset U'$.
Let $x,x'\in B_r(x_0)$, and let $y$ be as above.
Setting
\[
\rho := \rho(x)
\quad \text{ and } \quad
h := x'-x,
\]
we have for $s\in[-s_0,s_0]$
\begin{align*}
g((\rho+s)x')
&\ge
g(\rho x)
+\langle y,(\rho+s)x'-\rho x\rangle
-C_1|(\rho+s)x'-\rho x|^2 \\[5pt]
&=
f(\rho x)
+(\rho+s)\langle y,h\rangle
+s\langle y,x\rangle
-C_1|(\rho+s)x'-\rho x|^2 \\[5pt]
&\ge
(\rho+s)f(x)
+\rho\langle y,h\rangle
+s\varepsilon f(x)
-C_2(|h|^2+s^2),
\end{align*}
where $C_2$ depends only on $C_1$ and $|x'|\leq |x_0|+r$,
and we have used the $1$-homogeneity of $f$ and \eqref{eqcb}.
Using the $1$-homogeneity of $f$ again, we obtain form this estimate that
\begin{align*}
f((\rho+s)x') - g((\rho+s)x')
&\le
(\rho+s)\bigl(f(x')-f(x)\bigr)
-\rho\langle y,h\rangle
-s\varepsilon f(x)
+C_2(|h|^2+s^2) \\[5pt]
&=
\rho\langle \nabla f(x),h\rangle
-\rho\langle y,h\rangle
-s\varepsilon f(x)
+C_3(|h|^2+s^2),
\end{align*}
where $C_3$ depends only on $C_2$, the Lipschitz constants of $f$ and $\nabla f$ on $B_r(x_0)$, and the supremum of $\rho$ on $B_r(x_0)$.
Now set
\[
z := \frac{\rho}{\varepsilon f(x)}
\bigl(\nabla f(x)-y\bigr)
\quad \text{ and } \quad
s' := \langle z,h\rangle ,
\]
and note that $|z|\le C_4$ and $|s'|\le C_4|h|$ with $C_4$ depending only on $\rho$, $\ep$, $f$ and $|y|\leq C_1$. Set $s'':=s-s'$. We obtain
\begin{align}
f((\rho+s)x')-g((\rho+s)x')
&\le
-s''\varepsilon f(x)
+
C_3\bigl(|h|^2+s^2\bigr)  \nonumber
\\[5pt]
&\le
-s''\varepsilon f(x)
+
C_3(1+2C_4^2)|h|^2
+
2C_3(s'')^2.
\label{eq:ca}
\end{align}
Now let $r'<r$ such that $C_3(1+2C_4^2)\delta^2
>
2C_3(2C_5\delta^2)^2$ and $C_4\delta+2C_5\delta^2\le s_0$
for all $0\le\delta\le 2r'$,
where
\[
C_5
:=
\frac{C_3(1+2C_4^2)}
{\varepsilon\inf_{B_r(x_0)} f}.
\]
Now suppose that $x,x'\in B_{r'}(x_0)$ and that $y$ is as above. Define also $s'$ as above, and set $s''=2C_5|h|^2$ and $s:=s'+s''$. By definition of $r'$, we have $s\in [-s_0,s_0]$. Thus, by \eqref{eq:ca},
\[
f((\rho+s)x')-g((\rho+s)x') \leq -C_3(1+2C_4)|h|^2+2C_3(2C_5|h|^2)^2
<
0.
\]
This combines with our assumption on the function $\rho$ to yield that 
\[
\rho(x')
<
\rho(x)+s'+s''
=
\rho(x)
+
\langle z,x'-x\rangle
+
2C_5|x'-x|^2.
\]
This shows that $\rho$ is semiconcave on $B_{r'}(x_0)$, and this concludes the proof.
\end{proof}

\begin{proof}[Proof of Proposition \ref{li}]
    Since $\rho$ and $\lambda$ are $(-1)$-homogeneous, it costs no generality to assume $\rho(v_*)=1$. 

    Let $f(v):=|v|_g$, $v\in \op{int}(\C_x)$. We know that $d$ is locally semiconvex on $I^+$. Hence, also $g(v):=d(x,\exp_x(v))$ is locally semiconvex in a neighbourhood $U\subseteq \op{int}(\C_x)$ of $v_*\in T_xM$. Being $\rho$ continuous at $v_*$, we can pick an open neighbourhood $U'\subseteq U$ of $v_*$ such that $\rho(v)v\in U$ for all $v\in U'$.

    Taking into account Lemma \ref{lj}, we may also choose $U'$ even smaller to achieve that there exists $\ep>0$ such that, for any $v\in U'$, there is $w$ future directed causal with $|w|_g=|v|_g$, $c_v(\rho(v))=c_w(\rho(w))$, and such that
    \[
        \Big\langle \frac{\dot c_w(\rho(w))}{|w|_g}, \frac{\dot c_v(\rho(v))}{|v|_g}\Big \rangle_g \leq -(1+\ep).
    \]
    Note that $-\frac{\dot c_w(\rho(w))}{|w|_g}$ is a sub-gradient for $d(x,\cdot)$ at $c_v(\rho(v))=\exp_x(\rho(v)v)$ by \eqref{eqsub}. In particular (since $\rho(v)v\in U$), 
    \[
        p:= -\Big\langle \frac{\dot c_w(\rho(w))}{|w|_g},\cdot\Big\rangle \circ d_{\rho(v)v}\exp_x\in \partial^- g(\rho(v)v)
    \]
    and 
    \[
        p(v)\geq (1+\ep)|v|_g=(1+\ep)f(v).
    \]
    The preceding lemma yields that $\rho$ is locally semiconcave on $U'$.
\end{proof}

\section{Euclidean vs hyperbolic point of view}\label{App: C}

\begin{proof}[Proof of Lemma \ref{ll6}]
    Fix $x,\tilde x\in \Hy$ and consider the unique unit speed $\eta$-geodesic $c:[0,d_{\Hy}(x,\tilde x)]\to M$ from $x$ to $\tilde x$. 
    
    We first claim that the hyperbolic angle $\phi(c(t))$ never exceeds $\phi_{max}$. Indeed, let us write $c$ in the form 
    \[
        c(t):=\cosh(t)x+\sinh(t)v
    \]
    for some unit tangent vector $v\in T_x{\Hy}$.
    Let $c_1(t)$ be the first component of $c(t)$, so that $\phi(t)=\cosh^{-1}(c_1(t))$ is the hyperbolic angle of $c(t)$. 
    Since $\ddot c_1(t)=c_1(t)>0$, $c_1$ is convex, so that it attains its maximum value at $t=0$ or $t=d_{\Hy}(x,\tilde x)$. Being $\cosh$ non-decreasing, we deduce
    \begin{align}
        \phi(t) \leq \phi_{max}, \label{eqaq}
    \end{align}
    proving our claim.

    Now we write $c$ in the form
    \[
        c(t) =
        \begin{pmatrix} 
            \cosh(\phi(t)) 
            \\ 
            \sinh(\phi(t)) y(t)
        \end{pmatrix} 
        \text{ and } 
        \dot c(t) =
        \begin{pmatrix} 
            \sinh(\phi(t)) \dot \phi(t)  
            \\ 
            \cosh(\phi(t)) \dot \phi(t) y_t +\sinh(\phi(t)) \dot y(t)
        \end{pmatrix}, 
    \]
    where $\phi(t)$ is the hyperbolic angle of $c(t)$ and $y(t)\in S^{n-2}\subseteq \R^{n-1}$ is a Euclidean unit vector.

    Since $y(t)$ is orthogonal to $\dot y(t)$ w.r.t.\ the Euclidean scalar product in $\R^{n-1}$, we compute 
    \begin{gather*}
        |\dot c(t)|^2_{\eta} = \dot \phi(t)^2+\sinh^2(\phi(t)) |\dot y(t)|^2 
        \\[5pt]
        \text{and} \qquad |\dot c(t)|^2_{\op{euc}} = (\cosh^2(\phi(t))+\sinh^2(\phi(t)))\dot \phi(t)^2+\sinh^2(\phi(t)) |\dot y(t)|^2.
    \end{gather*}
    These identities combine with the estimate $\sinh \leq \cosh$ and \eqref{eqaq} to yield $|\dot c(t)|_{euc} \leq 2\cosh(\phi_{max})|\dot c(t)|_{\eta}$. We can therefore integrate and conclude
    \[
        |x-\tilde x|_{\op{euc}} 
        \leq 2\cosh(\phi_{max}) \int_0^{d_{\Hy}(x,\tilde x)} |\dot c(t)|_{\eta}\, dt
        = 2\cosh(\phi_{max}) d_{\Hy}(x,\tilde x),
    \]
    where we have used the fact that $c$ is $\eta$-minimizing in the second step. 
\end{proof}

\section*{Acknowledgements}
Most of this manuscript was written during my research stay with Robert McCann at the University of Toronto. I am grateful for the hospitality in Toronto and especially to Robert McCann for many fruitful discussions on this subject, in particular for the suggestion of Proposition \ref{li}.
I would also like to thank Stefan Suhr for the many discussions and valuable feedback, as well as for his reading of parts of the manuscript.

\bibliography{Lipschitz} 

@misc{Baer,
  author       = {Christian Bär},
  title        = {{Lecture Notes on Lorentzian Geometry}},
  note         = {\url{https://www.math.uni-potsdam.de/fileadmin/user_upload/Prof-Geometrie/Dokumente/Lehre/Lehrmaterialien/skript-LorGeo.pdf}, 30.07.2025},
}

@BOOK{DoCarmo,
  title     = "Riemannian Geometry",
  author    = "Do Carmo, Manfredo P",
  publisher = "Birkhauser Boston",
  year      =  2013,
  address   = "Secaucus, NJ, USA",
  language  = "en"
}

@book {Ehrlich,
    AUTHOR = {Beem, John K. and Ehrlich, Paul E. and Easley, Kevin L.},
     TITLE = {Global {L}orentzian geometry},
    SERIES = {Monographs and Textbooks in Pure and Applied Mathematics},
    VOLUME = {202},
   EDITION = {Second},
 PUBLISHER = {Marcel Dekker, Inc., New York},
      YEAR = {1996},
     PAGES = {xiv+635},
      ISBN = {0-8247-9324-2},
   MRCLASS = {53C50 (53-02 83-02)},
  MRNUMBER = {1384756},
MRREVIEWER = {Peter\ R.\ Law},
}

@BOOK{ONeill,
  title     = "{Semi-Riemannian} geometry with applications to relativity:
               Volume 103",
  author    = "O'Neill, Barrett",
  publisher = "Academic Press",
  year      =  1983,
  address   = "San Diego, CA, USA",
  language  = "en"
}

@article {McCann2,
    AUTHOR = {McCann, Robert J.},
     TITLE = {Displacement convexity of {B}oltzmann's entropy characterizes
              the strong energy condition from general relativity},
   JOURNAL = {Camb. J. Math.},
  FJOURNAL = {Cambridge Journal of Mathematics},
    VOLUME = {8},
      YEAR = {2020},
    NUMBER = {3},
     PAGES = {609--681},
      ISSN = {2168-0930,2168-0949},
   MRCLASS = {53C50 (49Q22 53C21 58Z05 82C35 83C99)},
  MRNUMBER = {4192570},
       DOI = {10.4310/CJM.2020.v8.n3.a4},
       URL = {https://doi.org/10.4310/CJM.2020.v8.n3.a4},
}

@article{Minguzzi,
  title={Lorentzian causality theory},
  author={Ettore Minguzzi},
  journal={Living Reviews in Relativity},
  year={2019},
  volume={22},
  pages={1-202},
  url={https://api.semanticscholar.org/CorpusID:195343105}
}

@book {Klingenberg,
    AUTHOR = {Klingenberg, Wilhelm},
     TITLE = {Riemannian geometry},
    SERIES = {De Gruyter Studies in Mathematics},
    VOLUME = {1},
 PUBLISHER = {Walter de Gruyter \& Co., Berlin-New York},
      YEAR = {1982},
     PAGES = {x+396},
      ISBN = {3-11-008673-5},
   MRCLASS = {53-02 (53C22 58-02 58E10 58F17)},
  MRNUMBER = {666697},
MRREVIEWER = {Alan\ West},
}

@article{Metsch3,
      title={On the locus of multiple maximizing geodesics on a globally hyperbolic spacetime}, 
      author={Alec Metsch},
      journal={arXiv preprint arXiv:2507.22737},
      year={2025},
      eprint={2507.22737},
      archivePrefix={arXiv},
      primaryClass={math.OC},
      url={https://arxiv.org/abs/2507.22737}, 
}

@article {Figalli/Rifford/Villani,
    AUTHOR = {Figalli, A. and Rifford, L. and Villani, C.},
     TITLE = {Tangent cut loci on surfaces},
   JOURNAL = {Differential Geom. Appl.},
  FJOURNAL = {Differential Geometry and its Applications},
    VOLUME = {29},
      YEAR = {2011},
    NUMBER = {2},
     PAGES = {154--159},
      ISSN = {0926-2245,1872-6984},
   MRCLASS = {53C20 (53C22)},
  MRNUMBER = {2784296},
MRREVIEWER = {Luca\ Granieri},
       DOI = {10.1016/j.difgeo.2011.02.002},
       URL = {https://doi.org/10.1016/j.difgeo.2011.02.002},
}

@article {Itoh/Tanaka,
    AUTHOR = {Itoh, Jin-ichi and Tanaka, Minoru},
     TITLE = {The {L}ipschitz continuity of the distance function to the cut
              locus},
   JOURNAL = {Trans. Amer. Math. Soc.},
  FJOURNAL = {Transactions of the American Mathematical Society},
    VOLUME = {353},
      YEAR = {2001},
    NUMBER = {1},
     PAGES = {21--40},
      ISSN = {0002-9947,1088-6850},
   MRCLASS = {53C20 (26A16 28A78 53C22)},
  MRNUMBER = {1695025},
MRREVIEWER = {James\ J.\ Hebda},
       DOI = {10.1090/S0002-9947-00-02564-2},
       URL = {https://doi.org/10.1090/S0002-9947-00-02564-2},
}

@article {Li/Nirenberg,
    AUTHOR = {Li, Yanyan and Nirenberg, Louis},
     TITLE = {The distance function to the boundary, {F}insler geometry, and
              the singular set of viscosity solutions of some
              {H}amilton-{J}acobi equations},
   JOURNAL = {Comm. Pure Appl. Math.},
  FJOURNAL = {Communications on Pure and Applied Mathematics},
    VOLUME = {58},
      YEAR = {2005},
    NUMBER = {1},
     PAGES = {85--146},
      ISSN = {0010-3640,1097-0312},
   MRCLASS = {35F20 (53C60)},
  MRNUMBER = {2094267},
MRREVIEWER = {Jie\ Yang},
       DOI = {10.1002/cpa.20051},
       URL = {https://doi.org/10.1002/cpa.20051},
}

@article {Castelpietra/Rifford,
    AUTHOR = {Castelpietra, Marco and Rifford, Ludovic},
     TITLE = {Regularity properties of the distance functions to conjugate
              and cut loci for viscosity solutions of {H}amilton-{J}acobi
              equations and applications in {R}iemannian geometry},
   JOURNAL = {ESAIM Control Optim. Calc. Var.},
  FJOURNAL = {ESAIM. Control, Optimisation and Calculus of Variations},
    VOLUME = {16},
      YEAR = {2010},
    NUMBER = {3},
     PAGES = {695--718},
      ISSN = {1292-8119,1262-3377},
   MRCLASS = {49L25 (35F20 53C22)},
  MRNUMBER = {2674633},
MRREVIEWER = {Marino\ Belloni},
       DOI = {10.1051/cocv/2009020},
       URL = {https://doi.org/10.1051/cocv/2009020},
}

@article {Whitehead,
    AUTHOR = {Whitehead, J. H. C.},
     TITLE = {On the covering of a complete space by the geodesics through a
              point},
   JOURNAL = {Ann. of Math. (2)},
  FJOURNAL = {Annals of Mathematics. Second Series},
    VOLUME = {36},
      YEAR = {1935},
    NUMBER = {3},
     PAGES = {679--704},
      ISSN = {0003-486X,1939-8980},
   MRCLASS = {99-04},
  MRNUMBER = {1503245},
       DOI = {10.2307/1968651},
       URL = {https://doi.org/10.2307/1968651},
}

@article {Gluck/Singer,
    AUTHOR = {Gluck, Herman and Singer, David},
     TITLE = {Scattering of geodesic fields. {I}},
   JOURNAL = {Ann. of Math. (2)},
  FJOURNAL = {Annals of Mathematics. Second Series},
    VOLUME = {108},
      YEAR = {1978},
    NUMBER = {2},
     PAGES = {347--372},
      ISSN = {0003-486X},
   MRCLASS = {53C20 (57R70)},
  MRNUMBER = {506991},
MRREVIEWER = {John\ Bolton},
       DOI = {10.2307/1971170},
       URL = {https://doi.org/10.2307/1971170},
}

@article {Hebda,
    AUTHOR = {Hebda, James J.},
     TITLE = {Metric structure of cut loci in surfaces and {A}mbrose's
              problem},
   JOURNAL = {J. Differential Geom.},
  FJOURNAL = {Journal of Differential Geometry},
    VOLUME = {40},
      YEAR = {1994},
    NUMBER = {3},
     PAGES = {621--642},
      ISSN = {0022-040X,1945-743X},
   MRCLASS = {53C20},
  MRNUMBER = {1305983},
MRREVIEWER = {Sheila\ Carter},
       URL = {http://projecteuclid.org/euclid.jdg/1214455780},
}

@article {Itoh,
    AUTHOR = {Itoh, Jin-ichi},
     TITLE = {The length of a cut locus on a surface and {A}mbrose's
              problem},
   JOURNAL = {J. Differential Geom.},
  FJOURNAL = {Journal of Differential Geometry},
    VOLUME = {43},
      YEAR = {1996},
    NUMBER = {3},
     PAGES = {642--651},
      ISSN = {0022-040X,1945-743X},
   MRCLASS = {53C20 (28A75)},
  MRNUMBER = {1412679},
MRREVIEWER = {James\ J.\ Hebda},
       URL = {http://projecteuclid.org/euclid.jdg/1214458326},
}

@article {Cordero/McCann/Schmuckenschlaeger,
    AUTHOR = {Cordero-Erausquin, Dario and McCann, Robert J. and
              Schmuckenschl\"ager, Michael},
     TITLE = {A {R}iemannian interpolation inequality \`a{} la {B}orell,
              {B}rascamp and {L}ieb},
   JOURNAL = {Invent. Math.},
  FJOURNAL = {Inventiones Mathematicae},
    VOLUME = {146},
      YEAR = {2001},
    NUMBER = {2},
     PAGES = {219--257},
      ISSN = {0020-9910,1432-1297},
   MRCLASS = {58E35 (28C99 60E15)},
  MRNUMBER = {1865396},
MRREVIEWER = {C\'edric\ Villani},
       DOI = {10.1007/s002220100160},
       URL = {https://doi.org/10.1007/s002220100160},
}

@article {Sakai,
    AUTHOR = {Sakai, Takashi},
     TITLE = {On the structure of cut loci in compact {R}iemannian symmetric
              spaces},
   JOURNAL = {Math. Ann.},
  FJOURNAL = {Mathematische Annalen},
    VOLUME = {235},
      YEAR = {1978},
    NUMBER = {2},
     PAGES = {129--148},
      ISSN = {0025-5831,1432-1807},
   MRCLASS = {53C20},
  MRNUMBER = {500710},
MRREVIEWER = {H.\ R.\ Gluck},
       DOI = {10.1007/BF01405010},
       URL = {https://doi.org/10.1007/BF01405010},
}

@article {Takeuchi,
    AUTHOR = {Takeuchi, Masaru},
     TITLE = {On conjugate loci and cut loci of compact symmetric spaces.
              {I}},
   JOURNAL = {Tsukuba J. Math.},
  FJOURNAL = {Tsukuba Journal of Mathematics},
    VOLUME = {2},
      YEAR = {1978},
     PAGES = {35--68},
      ISSN = {0387-4982,2423-821X},
   MRCLASS = {53C35},
  MRNUMBER = {531960},
MRREVIEWER = {Bang-yen\ Chen},
       DOI = {10.21099/tkbjm/1496158504},
       URL = {https://doi.org/10.21099/tkbjm/1496158504},
}
\bibliographystyle{acm}
\end{document}